\documentclass[11pt,reqno,a4paper]{amsart}
\usepackage{a4wide}
\usepackage{amssymb}
\usepackage{graphicx}
\usepackage{subfig}
\usepackage{psfrag}
\usepackage{bbm}
\numberwithin{equation}{section}

\advance \vsize by 1cm
\newcommand{\Rr}{\mathbb{R}}

\newcommand{\Cc}{\mathbb{C}}
\renewcommand{\epsilon}{\varepsilon}
\newcommand*{\norm}[1]{\left\lVert#1\right\rVert}%norme

\DeclareMathOperator{\e}{e}
\newcommand{\integ}[3]{%
\ensuremath{\displaystyle{\int^{#2}_{#1} #3}}}
\newcommand{\somme}[3]{%
\ensuremath{\displaystyle{\sum^{#2}_{#1} #3 }}}

\theoremstyle{plain}
\newtheorem{theo}{Theorem}[section]
\newtheorem{lem}[theo]{Lemma}
\newtheorem{prop}[theo]{Proposition}
\newtheorem{coro}[theo]{Corollary}
\theoremstyle{definition}
\newtheorem{rem}[theo]{Remark}

\begin{document}
%\dedicatory{Dedicated to }
\author{Julien Dambrine \and Morgan Pierre \and Germain Rousseaux}
\title[A determination of optimal ship forms]{A theoretical and numerical determination of optimal ship forms based on Michell's wave resistance}

\begin{abstract}We determine the parametric hull  of a given volume which minimizes the total water resistance for a given speed of the ship. The total resistance is  the sum of Michell's wave resistance and of the  viscous resistance,  approximated by assuming a constant viscous drag coefficient.  We prove that the optimized hull exists, is unique, symmetric, smooth and that  it depends continuously on the speed. Numerical simulations show the efficiency of the approach, and complete the theoretical results.   
\end{abstract} 
\address{Julien Dambrine and Morgan Pierre, 
Universit\'e de Poitiers,
Laboratoire de Ma\-th\'e\-ma\-tiques et Applications UMR CNRS 7348,
 T\'el\'eport 2 - BP 30179,
 Boulevard Marie et Pierre Curie,
86962 Futuroscope Chasseneuil,
France.}
%email: \texttt{pierre@math.univ-poitiers.fr}}
\address{Germain Rousseaux, Institut Pprime UPR 3346,
D\'epartement Fluides, Thermique, Combustion,
CNRS - Universit\'e de Poitiers - ENSMA,
SP2MI - T\'el\'eport 2,
11 Boulevard Marie et Pierre Curie, BP 30179,
86962 Futuroscope Chasseneuil Cedex,
France.}
%
%E-mail : germain.rousseaux@univ-poitiers.fr
%,}
\maketitle
\noindent
%\today


%{\bf AMS Classification}

\section{Introduction}

 The resistance of water to the motion of a ship is  traditionally represented as the sum of  two terms,   the wave resistance and the viscous resistance (which corresponds itself to the sum of the frictional and eddy resistance). Michell's thin-ship theory~\cite{Michell,Michelsen}  provides an explicit formula of  the wave resistance for a given speed and for a hull expressed in parametric form, with parameters in the region of the plane of symmetry. It is therefore a natural question    to  search the hull of a given volume which minimizes Michell's wave resistance for a given speed. Unfortunately, this problem is known to be ill-posed~\cite{Kostyukov,Sretensky_1936}: it is underdetermined, so that additional constraints should be imposed in order to provide a  solution.  The latter approach has been successfully performed by several authors,  from a theoretical and computational point of view, starting in the 1930's with Weinblum (see~\cite{Weinblum} and references in~\cite{Kostyukov}), Pavlenko~\cite{Pavlenko}, until more recently~\cite{Gotman,Guilloton,Higuchi_Maruo_1979,Hsiung_1981,Hsiung_Shenyyan_1984,Maruo_Bessho}. 

In this paper, instead of using Michell's formula alone as an optimization criterion,  we propose to use the total resistance, by adding  to Michell's wave resistance a  term approximating the viscous resistance; this term is obtained  by assuming a constant viscous drag coefficient in the framework of the thin-ship approximation. Our approach, which results in quadratic programming,  has already been considered from a numerical point of view in~\cite{Lian}. From a theoretical point of view, a similar approach has been made in~\cite{Michalski_etal}, but the additional term was more complexe to deal with, and the analysis was therefore incomplete. 

Here,  we prove that minimizing this total resistance for a given speed, among the parametric hulls having a fixed volume and a fixed domain of parameters, is a well-posed problem. We also prove that the optimized hull is a smooth and symmetric form, which  depends continuously on the speed. Our theoretical results also include the case where Michell's wave resistance for an infinite fluid is replaced by Sretensky's formula in an infinitely deep and laterally confined fluid~\cite{Sretensky_1936}. For the numerical simulations, made with the \texttt{Scilab} software\footnote{\texttt{Scilab} is freely available at \textsf{http://www.scilab.org/}}, we use an efficient  $Q^1$ finite element discretization of the problem (use of ``tent functions'').  We recover results similar to those in~\cite{Lian};  in particular, for moderate values of the velocity,  we obtain the famous bulbous bow which reduces the wave resistance~\cite{Inui,Kostyukov}. In addition, we give numerical evidence that using Michell's wave resistance as an optimization criterion results in an ill-posed problem, and we obtain a theoretical lower bound on the degrees of freedom that should be used in order to minimize efficiently the wave resistance.

Of course,  nowadays,  computational fluid dynamics (CFD) provide more precise tools for ship hull optimization (see, for instance,~\cite{Hochkirch_Bertram,Mohammadi_Pironneau_2010,Park_Choi,Percival_etal_2001,Saha_etal,Zhang_etal_2009}). However, in spite of its well-known limitations (see~\cite{Gotman} for a review of these limitations), Michell's formula for the wave resistance remains a powerful tool for theoretical and computational purposes. The simplicity of our formulation allows us to obtain theoretical results which are at the present moment out of reach when considering the  full 3-dimensional incompressible Navier-Stokes equations. Moreover, our numerical approach is much faster than standard CFD computations.

The optimization problem is formulated in Section~\ref{sec2}. Well-posedness and related theoretical results are proved in Section~\ref{sec3}. Numerical methods are explained in   Section~\ref{sec4}, and the numerical results are given and commented in   Section~\ref{sec5}.   

\section{Formulation of the optimization problem}
\label{sec2}
Consider a ship moving with constant velocity on the surface of an unbounded  fluid.  A coordinated system fixed with respect to the ship is introduced.  The origin is located at midship in the center line plane, the $xy$-plane is the undisturbed water surface,   the positive $x$-axis is in the direction of motion and the $z$-axis is vertically downward.

 The hull is assumed to be symmetric with respect to the vertical $xz$-plane, with length   $L$ and   draft $T$. The immerged hull surface is represented by a  continuous nonnegative function 
$$y=f(x,z)\ge 0,\quad x\in [-L/2,L/2],\quad z\in [0,T],$$
with $f(\pm L/2,z)=0$ (for all $z$) and $f(x,T)=0$ (for all $x$). 
 
It is assumed that the fluid is incompressible, inviscid and  that the flow is irrotational. The effects of surface tension are neglected.  The motion has persisted long enough so that a steady state has been reached. Michell's theory~\cite{Michell} shows that the wave resistance can be computed by
\begin{equation}
\label{RM}
R_{Michell}=\frac{4\rho g^2}{\pi U^2}\int_1^\infty(I(\lambda)^2+J(\lambda)^2)\frac{\lambda^2}{\sqrt{\lambda^2-1}}d\lambda,
\end{equation}
with
\begin{equation}
\label{I}
I(\lambda)=\int_{-L/2}^{L/2}\int_{0}^T\frac{\partial f(x,z)}{\partial x}\exp\left(-\frac{\lambda^2 gz}{U^2}\right)\cos\left(\frac{\lambda g x}{U^2}\right)dxdz,
\end{equation}
\begin{equation}
\label{J}
J(\lambda)=\int_{-L/2}^{L/2}\int_0^T\frac{\partial f(x,z)}{\partial x}\exp\left(-\frac{\lambda^2 gz}{U^2}\right)\sin\left(\frac{\lambda g x}{U^2}\right)dxdz.
\end{equation}
Here, $U$ (in $\mathrm{m}\cdot \mathrm{s}^{-1}$) is the speed of the ship, $\rho$ (in $\mathrm{kg}\cdot\mathrm{m}^{-3}$) is the (constant) density of the fluid, and $g$ (in $\mathrm{m}\cdot\mathrm{s}^{-2}$) is the standard gravity. The double integrals $I(\lambda)$ and $J(\lambda)$ are in $\mathrm{m}^2$, and $R_{Michell}$ (in Newton) has the dimension of a force. The integration parameter $\lambda$ has no dimension: it can be interpreted as $\lambda=1/\cos\theta$, where $\theta$ is the angle at which the wave is propagating~\cite{Havelock}. 

In order to derive formula~\eqref{RM}, Michell used a linear theory and made additional assumptions known as the ``thin ship theory'' (see~\cite{Michelsen} for details). In particular, it is assumed that the angles made by the hull surface with the longitudinal plane of symmetry are small, i.e. 
\begin{equation}
\label{Dfpetit}
|\partial_xf|\ll 1\quad\mbox{ and }\quad |\partial_zf|\ll 1\quad\mbox{ in }\quad [-L/2,L/2]\times[0,T].
\end{equation}

For simplicity, we define 
$$v=g/U^2>0\quad\mbox{ and }\quad T_f(v,\lambda)=I(\lambda)-iJ(\lambda),$$
where $I$ and $J$ are given by~\eqref{I}-\eqref{J}. Then 
\begin{equation}
\label{defT}
T_f(v,\lambda)=\int_{-L/2}^{L/2}\int_0^T\partial_xf(x,z)e^{-\lambda^2 v z}e^{-i\lambda v x}dxdz,
\end{equation}
and $R_{Michell}$ can be written
\begin{equation}
\label{defR}
R(v,f)=\frac{4\rho g v}{\pi }\int_1^\infty\left|T_f(v,\lambda)\right|^2\frac{\lambda^2}{\sqrt{\lambda^2-1}}d\lambda.
\end{equation}
 The number $v$ (in $\mathrm{m}^{-1}$) is known as the Kelvin wave number for the transverse waves in deep water~\cite{Kelvin}. Notice that $\rho$ and $g$ are fixed, so $R$ depends only on $v$, i.e. the speed $U$, and on $f$, i.e. the form of the hull.

In view of numerical computations, we let $\Lambda\gg 1$ denote a real number and we replace $R(v,f)$ by the functional
\begin{equation}
\label{RLambda1}
(v,f)\mapsto\frac{4\rho g v}{\pi }\int_1^\Lambda\left|T_f(v,\lambda)\right|^2\frac{\lambda^2}{\sqrt{\lambda^2-1}}d\lambda.
\end{equation}
 For the numerical computation, we actually use a numerical integration formula of the form
\begin{equation}
\label{RLambda2}
\frac{4\rho g v}{\pi}\int_1^\Lambda\left|T_f(v,\lambda)\right|^2\frac{\lambda^2}{\sqrt{\lambda^2-1}}d\lambda\approx \frac{4\rho g v}{\pi}\sum_{j=1}^{J^\star}\omega_j\left|T_f(v,\lambda_j)\right|^2,
\end{equation}
with positive weights $\omega_j>0$, and with nodes $\lambda_j\in[1,\Lambda]$, $j=1,2,\ldots, J^\star$, where $J^\star$ is a well-chosen  positive integer (see~\eqref{num-6}). 

 In order to take into account the two formulations~\eqref{RLambda1} and~\eqref{RLambda2} in our analysis, we consider more generally a wave resistance of the form
\begin{equation}
\label{defRLambda}
R^\Lambda(v,f)=\frac{4\rho g v}{\pi }\int_1^\Lambda\left|T_f(v,\lambda)\right|^2d\mu(\lambda),
\end{equation}
where $\mu$ is a nonnegative and finite borelian  measure on $[1,\Lambda]$. Such a formulation also includes (a truncation of) Sretensky's summation formula for the wave resistance of a thin ship in a laterally confined and infinitely deep fluid~\cite{Sretensky_1936}.

 We point out that our well-posedness result  holds also for the functional $R(v,f)$ defined by~\eqref{defR} or for Sretensky's formula~\cite{Sretensky_1936} (see Remark~\ref{rem3.2}), but otherwise, setting $\Lambda<\infty$ simplifies the analysis, because the integral $\int_1^\infty\lambda^2(\lambda^2-1)^{-1/2}d\lambda$ diverges at $\infty$.

Let us turn now to the term representing the viscous resistance, or {\it viscous drag}~\cite{Molland_etal}. It reads 
$$R_{drag}=\frac{1}{2}  \, \rho U^2 \, C_d \, A \, ,$$ 
where $C_d$ is the viscous drag coefficient (which at some extent can be considered constant within the family of slender bodies), and $A$ is the surface area of the ship's wetted hull. When the graph of $f$ represents the ship's hull, $A$ is given by:
\begin{equation}
A=2 \integ{\Omega}{}{\sqrt{1 + | \nabla f(x,z) |^2} \, \mathrm{d}x \mathrm{d}z} \, ,
\end{equation}
 where here and below,  $\Omega=(-L/2,L/2)\times(0,T)$. 
When the ship is slender (\textit{i.e.}  $| \nabla f |$ uniformly small, see~\eqref{Dfpetit}), one can give a good approximation of the above integral by performing a Taylor expansion of $\sqrt{1 + | \nabla f |^2}$ at first order, for small values of $| \nabla f |^2$:
\begin{equation}
A/2 = 1 + \frac{1}{2}\integ{\Omega}{}{| \nabla f (x,z)|^2\, \mathrm{d}x \mathrm{d}z} + o(|| \nabla f ||_{\infty}) \, .
\end{equation}

The approximation of the viscous drag for small $\nabla f$ then reads:
$$R_{drag}= \rho U^2 \, C_d \, \left(1 + \frac{1}{2}\integ{\Omega}{}{| \nabla f (x,z)|^2\, \mathrm{d}x \mathrm{d}z} \right) \, .$$
Minimizing $R_{drag}$ is the same as minimizing the following quantity:
$$R_{drag}^*=\frac{1}{2}  \, \rho U^2 \, C_d \, \integ{\Omega}{}{| \nabla f (x,z)|^2\, \mathrm{d}x \mathrm{d}z} \, .$$
By setting 
\begin{equation}
\label{defeps}
\epsilon=\frac{1}{2}  \, \rho U^2 \, C_d,
\end{equation}
 we obtain 
$$R_{drag}^*=\epsilon \, \integ{\Omega}{}{| \nabla f (x,z)|^2\, \mathrm{d}x \mathrm{d}z} \, .$$
The parameter $\epsilon$  (in $\mathrm{Pa}$) is positive; it can be interpreted as a dynamical pressure, as in Bernoulli's law. 

The total water resistance functional $N^{\Lambda,\epsilon}(v,\cdot)$ is the sum of the wave resistance and of the viscous drag $R^\star_{drag}$: 
$$N^{\Lambda,\epsilon}(v,f):= R^\Lambda(v,f)+\epsilon\int_\Omega|\nabla f(x,z)|^2dxdz,$$
where $R^\Lambda$ is defined by~\eqref{defRLambda}. We will  minimize $N^{\Lambda,\epsilon}(v,\cdot)$, among admissible functions. 
Notice that  the additional term $\int_\Omega|\nabla f(x,z)|^2dxdz$ is isotropic, i.e. that no direction is priviledge in the $(x,z)$ plane.  This term   guarantees that the derivatives of a minimizer $f$ are defined in the space $L^2(\Omega)$ of square integrable function. Since we seek a minimizer, the additional term is small,  thus fulfilling the thin ship assumptions~\eqref{Dfpetit} in an integral sense (rather than pointwise).

The function space is now clear from the additional term, and we therefore introduce the space
$$H=\left\{f\in H^1(\Omega)\ : f(\pm L/2,\cdot)=0\mbox{ and }f(\cdot,T)=0\mbox{ in the sense of traces}\right\},$$
where $H^1(\Omega)$ denotes the standard $L^2$-Sobolev space (see, for instance,~\cite{Brezis}).  
$H$ is a closed subspace of $H^1(\Omega)$, so it  is a Hilbert space for the standard $H^1(\Omega)$-norm. We recall that 
$$f\mapsto \int_\Omega|\nabla f(x,z)|^2dxdz$$ 
is a norm on $H$, which is equivalent to the standard $H^1(\Omega)$-norm~\cite{Brezis}. This is due to the boundary values imposed in the definition of $H$. 

Let $V>0$ be the (half-)volume of an immerged hull. The set of admissible functions is the closed convex subset of $H$ defined by
$$C_V=\left\{f\in H\ :\ \int_\Omega f(x,z)dxdz=V\mbox{ and }f\ge 0\ \mbox{a.e. in }\Omega \right\}.$$

Our {\bf optimization problem $\mathcal{P}^{\Lambda,\epsilon}$} reads: for a given  Kelvin wave number $v$ and for a given volume $V>0$, find the function $f^\star$ which minimizes $N^{\Lambda,\epsilon}(v,f)$ among functions $f\in C_V$. 
\section{Resolution of the optimization problem}
\label{sec3}
\subsection{Well-posedness of the problem}
Unless otherwise stated, the parameters $\rho>0$, $g>0$, $V>0$, $\Lambda>0$,  $v>0$  and $\epsilon>0$ are fixed. We have:
\begin{theo}
\label{theowellposed}
Problem $\mathcal{P}^{\Lambda,\epsilon}$ has a unique solution $f^{\epsilon,v}\in C_V$. Moreover, $f^{\epsilon,v}$ is even with respect to $x$. 
\end{theo}
\begin{proof}The Hilbertian norm $f\mapsto N^{\Lambda,\epsilon}(v,f)$ is strictly convex on $H$, because $f\mapsto R^\Lambda(v,f)$ is convex and $f\mapsto \int_\Omega|\nabla f(x,z)|^2dxdz$ is strictly convex. Since the set $C_V$ is convex, any minimizer is unique. 

Let now $(f_n)$ be a minimizing sequence in $C_V$. Then $(f_n)$ is bounded in $H$, and we can extract a subsequence, still denoted $(f_n)$, such that $f_n$ converges weakly in $H$ to some $f$. Since $C_V$ is a convex set which is closed for the strong topology, $C_V$ is also closed for the weak topology (see, e.g.,~\cite{Brezis}), so  $f$ belongs to $C_V$. Since $\partial_xf_n\to \partial_xf$ weakly in $L^2(\Omega)$,  $T_{f_n}(v,\lambda)\to T_f(v,\lambda)$ for every $\lambda>0$. Thus, by Fatou's lemma,
$$R^\Lambda(v,f)\le \liminf_nR^\Lambda(v,f_n).$$
Moreover, the norm $\int_\Omega|\nabla \cdot|$ is lower semi-continuous for the weak $H^1$-topology. This implies that 
$$N^{\Lambda,\epsilon}(v,f)\le \liminf_n N^{\Lambda,\epsilon}(v,f_n),$$
and this shows the minimality of $f$. 

Next, we prove that the minimizer $f$ is even with respect to $x$. For a function $h\in H$, let  $\check{h}$ be the function in $H$ defined by $\check{h}(x,z)=h(-x,z)$ a.e. We notice that if $h\in C_V$, then $\check{h}\in C_V$. It is also easily seen that $R^\Lambda(v,\check{h})=R^\Lambda(v,h)$ for all $h\in H$ (use definitions~\eqref{I}-\eqref{J} and a change of variable $x\to -x$). Thus $\check{f}$ is a function in $C_V$ such that $N^{\Lambda,\epsilon}(v,\check{f})=N^{\Lambda,\epsilon}(v,f)$. By uniqueness of the minimizer, $\check{f}=f$. 
\end{proof}
\begin{rem}
\label{rem3.2}This well-posedness result and its proof are also valid  if one uses Michell's wave resistance $R(v,f)$ instead of $R^\Lambda(v,f)$ in the definition of the function $N^{\Lambda,\epsilon}$.  A similar statement holds for Sretensky's wave resistance~\cite{Sretensky_1936} in a laterally confined and infinitely deep fluid.
\end{rem}

The following assertion shows that when $\epsilon$ is small, our optimal solution is an approximate solution to the non-regularized  optimization problem, i.e. the problem of finding a ship with minimal wave resistance. 
\begin{prop}The minimum value $N^{\Lambda,\epsilon}(v,f^{\epsilon,v})$ tends to 
$$m^{\Lambda,v}:=\inf_{f\in C_V} R^\Lambda(v,f)$$
 as $\epsilon$ tends to $0$.  
\end{prop}
\begin{proof}Let $\beta>0$. By definition of the infimum, there exists $f\in C_V$ such that $m^{\Lambda,v}\le R^\Lambda(v,f)<m^{\Lambda,v}+\beta$. We choose $\epsilon_0>0$ small enough so that 
$$\epsilon_0\int_\Omega|\nabla f|^2<\beta.$$
We have  
$$N^{\Lambda,\epsilon_0}(v,f^{\epsilon_0,v})\le N^{\Lambda,\epsilon_0}(v,f)\le R^\Lambda(v,f)+\beta.$$
Thus, for all $\epsilon\in(0,\epsilon_0)$, we have
$$m^{\Lambda,v}<N^{\Lambda,\epsilon}(v,f^{\epsilon,v})\le N^{\Lambda,\epsilon}(v,f^{\epsilon_0,v})\le N^{\Lambda,\epsilon_0}(v,f^{\epsilon_0,v})\le m^{\Lambda,v}+2\beta.$$
Since $\beta>0$ is arbitrary, the proof is complete.
\end{proof}

\subsection{Continuity of the optimum with respect to $v$}
In this section, we prove  that $f^{\epsilon,v}$ changes continuously as the parameter $v$ changes.
We first notice:
\begin{prop}
\label{propbounded}
The linear operator $f\mapsto (\lambda\mapsto T_f(v,\lambda))$ is bounded from $H$ into $L^2([1,\Lambda],\mu)$. 
\end{prop}
\begin{proof} By the Cauchy-Schwarz inequality, 
\begin{eqnarray}
\left|T_f(v,\lambda)\right|^2&\le &\norm{\partial_x f}_{L^2(\Omega)}^2\int_\Omega \e^{-2\lambda^2 v z}dxdz\nonumber\\
&\le &\norm{\partial_x f}_{L^2(\Omega)}^2\frac{L}{2\lambda^2 v}.\label{aux1}
\end{eqnarray}
Thus, 
$$\norm{T_f(v,\lambda)}_{L^2([1,\Lambda],\mu)}\le \norm{\partial_x f}_{L^2(\Omega)}\left(\frac{L}{2v}\right)^{1/2}\mu([1,\Lambda])^{1/2},$$
and this  proves the claim, since $\mu([1,\Lambda])<\infty$ by assumption (cf.~\eqref{defRLambda}). 
\end{proof}
In particular, by definition~\eqref{defRLambda}, 
\begin{equation}
\label{propRLambda}
R^\Lambda(v,f)=\frac{4\rho g v}{\pi}\norm{T_f(v,\lambda)}_{L^2([1,\Lambda],\mu)}^2
\end{equation}
 is well defined for all $f\in H$, and $f\mapsto R^\Lambda(v,f)$ is a continuous nonnegative quadratic form on $H$. 

The following result will prove useful:
\begin{lem}
\label{lem1}
Let $(v_n)$ be a sequence of positive real numbers such that $v_n\to \bar{v}>0$, and let $(h_n)$ be a sequence in $H$ such that $h_n\to h$ weakly in $H$. Then $R^\Lambda(v_n,h_n)\to R^\Lambda(\bar{v},h)$.
\end{lem}
\begin{proof}
let $k(v,\lambda,x,z)=\e^{-\lambda^2 vz}\e^{-i\lambda v x}$ denote the kernel of $T_f$. By the mean value inequality, for all $\lambda\in[1,\Lambda]$, for all $x\in [-L/2,L/2]$ and for all $z\in [0,T]$, we have
$$|k(v_n,\lambda,x,z)-k(v,\lambda,x,z)|\le (\Lambda^2T+\Lambda L/2)|v_n-v|.$$
Thus, 
\begin{eqnarray*}
\left|T_{h_n}(v_n,\lambda)-T_{h_n}(\bar{v},\lambda)\right|&\le &(\Lambda^2T+\Lambda L/2)\int_\Omega |\partial_x h_n(x,z)|dxdz|v_n-v|\\
&\le &(\Lambda^2T+\Lambda L/2)\norm{\partial_x h_n}_{L^2(\Omega)}(LT)^{1/2}|v_n-v|,
\end{eqnarray*}
and so $T_{h_n}(v_n,\lambda)-T_{h_n}(\bar{v},\lambda)\to 0$ (in $\Rr$) as $n\to \infty$. Moreover, for all $\lambda\in [1,\Lambda]$, 
$$T_{h_n}(\bar{v},\lambda)- T_h(\bar{v},\lambda)\to 0$$
since $\partial_x h_n$ converges to $\partial_x h$ weakly in $L^2(\Omega)$. We deduce from the triangle inequality that for all $\lambda\in [1,\Lambda]$,
$$T_{h_n}(v_n,\lambda)\to T_h(\bar{v},\lambda).$$ 
Estimate~\eqref{aux1} in the proof of Proposition~\ref{propbounded} shows that $|T_{h_n}(v_n,\lambda)|$ is bounded by a constant independent of $n$ and $\lambda\in[1,\Lambda]$. Since the total measure $\mu$ is finite on $[1,\Lambda]$, we can apply Lebesgue's dominated convergence theorem, which yields
$$\norm{T_{h_n}(v_n,\lambda)}_{L^2([1,\Lambda],\mu)}\to \norm{T_{h}(\bar{v},\lambda)}_{L^2([1,\Lambda]),\mu)}.$$
The claim follows from~\eqref{propRLambda}. 
\end{proof}

 We can now state:
\begin{theo}
\label{theocont}
Let $\bar{v}>0$. Then $f^{\epsilon,v}$ converges strongly in $H$ to $f^{\epsilon,\bar{v}}$  as $v\to \bar{v}$.  
\end{theo}
\begin{proof}Let $(v_n)$ be a sequence of positive real numbers such that $v_n\to \bar{v}$. Our goal is to show that $f^{\epsilon,v_n}$ tends to $f^{\epsilon,\bar{v}}$ strongly in $H$.  

First, we claim that the sequence of functionals $(N^{\Lambda,\epsilon}(v_n,\cdot))_n$ $\Gamma$-converges to $N^{\Lambda,\epsilon}(\bar{v},\cdot)$ for the weak topology in $H$ (see, e.g.~\cite{Braides}). Indeed, let $(h_n)$ be a sequence in $H$ such that $h_n\to h$ weakly in $H$. Lemma~\ref{lem1} shows that $R^\Lambda(v_n,h_n)\to R^\Lambda(\bar{v},h)$. Using the lower semicontinuity of the norm in $H$, we deduce that 
\begin{equation}
\label{Nlsc}
 N^{\Lambda,\epsilon}(\bar{v},h)\le \liminf_nN^{\Lambda,\epsilon}(v_n,h_n) .
\end{equation}
Moreover, for any $h\in H$, using Lemma~\ref{lem1} again, we obtain 
\begin{equation}
\label{limsup}
N^{\Lambda,\epsilon}(\bar{v},h)=\lim_{n\to +\infty} N^{\Lambda,\epsilon}(v_n,h).
\end{equation}
This proves the claim. 

Next, we notice that the sequence $f^{\epsilon,v_n}$ is bounded in $H$, since for any choice of $h\in C_V$, we have 
$$\epsilon\frac{\rho g}{2 v_n}\int_\Omega|\nabla f^{\epsilon,v_n}|^2\le N^{\Lambda,\epsilon}(v_n,f^{\epsilon,v_n})\le N^{\Lambda,\epsilon}(v_n,h),$$
and the sequence $N^{\Lambda,\epsilon}(v_n,h)$ is bounded by~\eqref{limsup}. Thus, the sequence $f^{\epsilon,v_n}$ has an accumulation point (in $C_V$) for the weak topology in $H$; the $\Gamma$-convergence result (which is also valid in $C_V$) implies that any accumulation point is a minimizer of $N^{\Lambda,\epsilon}(\bar{v},\cdot)$, i.e. $f^{\epsilon,\bar{v}}$. Uniqueness of the minimizer implies that the whole sequence converges weakly in $H$ to $f^{\epsilon,\bar{v}}$.   

 Finally, we notice that 
 $$N^{\Lambda,\epsilon}(v_n,f^{\epsilon,v_n})\le N^{\Lambda,\epsilon}(v_n,f^{\epsilon,\bar{v}}),$$
and this, together with~\eqref{Nlsc} and~\eqref{limsup}, implies that 
$$\lim_{n\to +\infty}N^{\Lambda,\epsilon}(v_n,f^{\epsilon,v_n})= N^{\Lambda,\epsilon}(\bar{v},f^{\epsilon,\bar{v}}).$$
As a consequence,  by Lemma~\ref{lem1}, $\lim_{n\to +\infty}\int_\Omega|\nabla f^{\epsilon,v_n}|^2=\int_\Omega|\nabla f^{\epsilon,\bar{v}}|^2$, so $f^{\epsilon,v_n}$ converges strongly in $H$ to $f^{\epsilon,\bar{v}}$. This concludes the proof. 
\end{proof}
\begin{rem}If $\epsilon$ is a (strictly) positive and continuous function of $v$, then a simple adaptation of the proof above shows that Theorem~\ref{theocont} is still valid.  This is the case if $\epsilon=C_d\rho g/(2v)$ with $C_d$ constant, as in~\eqref{defeps}.
\end{rem}
\subsection{Regularity of the solution}
In this section, we prove the $W^{2,p}$ regularity of the solution for all $p<\infty$, by using the regularity of the non-constrained optimization problem.

 As a shortcut, we  define
$$a(u,w)=\epsilon\int_\Omega\nabla u\cdot\nabla w\, dxdz\qquad (u,w)\in H\times H,$$
so that $a$ is a continuous bilinear form on $H$; $a$ is also coercive, i.e. 
$$a(u,u)>0\quad\forall u\in H\setminus\{0\},$$
 and the  Hilbertian norm $a\mapsto a(u,u)^{1/2}$ is equivalent to the $H^1$-norm on $H$.  
Since the domain $\Omega$ is a rectangle, the  space $H$ is dense in $L^2(\Omega)$, and we have the continuous injections $H\hookrightarrow L^2(\Omega)\hookrightarrow H'$.  We can define $A$ the operator from $H$ into $H'$ such that 
\begin{equation}
\label{defA}
a(u,w)=\langle Au,w\rangle\qquad\forall u,w\in H,
\end{equation}
where  $\langle\cdot,\cdot\rangle$ denotes the duality product $H'\times H$.  

Let finally $H^+=\{f\in H\ :\ f\ge 0\mbox{ a.e. in }\Omega\}$ and
$$k(x,z,x',z')=\frac{4\rho g v^3}{\pi}\int_1^\Lambda\lambda^2\cos(\lambda v (x- x'))e^{-\lambda^2 v(z+z')}d\mu(\lambda).$$
Regularity of a minimizer is a consequence of  the Euler-Lagrange equation, which  reads:
\begin{prop}
\label{propEL}
The solution $f\equiv f^{\epsilon,v}$ of problem $\mathcal{P}^{\Lambda,\epsilon}$  satisfies the variational inequality
\begin{eqnarray*}
a(f,h-f)+\int_\Omega\left(\int_\Omega k(x,z,x',z')f(x',z')dx'dz'\right)(h-f)dxdz\\
\qquad\ge C\int_\Omega(h-f)dxdz\qquad\forall h\in H^+,
\end{eqnarray*}
for some constant $C\in\Rr$. 
\end{prop}
\begin{proof}Using the bilinear form $a$, and performing an integration by parts with respect to $x$ in formulas~\eqref{I}-\eqref{J},  for $h\in H$, we have
$$N^{\Lambda,\epsilon}(h)=a(h,h)+\frac{4\rho g v}{\pi}\int_1^\Lambda  I_h(\lambda,v)^2+J_h(\lambda,v)^2 d\mu(\lambda),$$
where 
$$I_h(\lambda,v)=\lambda v\int_\Omega h(x,z)e^{-\lambda^2 vz}\sin(\lambda v x)dxdz,$$
$$J_h(\lambda,v)=-\lambda v\int_\Omega h(x,z)e^{-\lambda^2 vz}\cos(\lambda v x)dxdz.$$

Let now $h\in H^+$ and set 
$$\varphi(t)=V(f+t(h-f))/(\int_\Omega f+t(h-f))\quad t\ge 0,$$
so that $\varphi(t)\in C_V$ for all $t\ge 0$ and $\varphi(0)=f$. Then $N^{\Lambda,\epsilon}(\varphi(t))\ge N^{\Lambda,\epsilon}(f)$, so
$$\frac{d}{dt}N^{\Lambda,\epsilon}(\varphi(t))_{|t=0}\ge 0.$$
Computing, we have $\varphi'(0)=(h-f)-f\int_\Omega(h-f)/V$ and 
\begin{eqnarray*}
\frac{d}{dt}N^{\Lambda,\epsilon}(\varphi(t))_{|t=0}&=&2a(f,\varphi'(0))\\
&&+\frac{8\rho g v}{\pi}\int_1^\Lambda I_f(\lambda,v)I_{\varphi'(0)}(\lambda,v)+J_f(\lambda,v)J_{\varphi'(0)}(\lambda,v)d\mu(\lambda).
\end{eqnarray*}
The expected variational inequality is obtained with the constant 
$$C=a(f,f)/V+\frac{4\rho g v}{\pi V}\int_1^\Lambda I_f(\lambda,v)^2+J_f(\lambda,v)^2d\mu(\lambda),$$
by an application of Fubini's theorem. 
\end{proof}
For sake of completeness, we recall the following classical result which relates the regularity of the constrained problem to the regularity of the unconstrained problem. Let $C^\infty_c(\Omega)$ denote the space of smooth functions with compact support in $\Omega$, and let
$$C^\infty_c(\Omega)^+=\{\varphi\in C^\infty_c(\Omega)\ :\ \varphi\ge 0\mbox{ in }\Omega\}.$$
We say that two elements $w,z\in H'$ satisfy $w\ge z$ if $\langle w,\varphi\rangle\ge \langle z,\varphi\rangle$ for all $\varphi\in C^\infty_c(\Omega)^+$. 
\begin{theo}
\label{theoreg1}
Let $w\in L^2(\Omega)$. The solution $f\in H^+$ of the variational problem 
\begin{equation}
\label{aux3.0}
a(f,h-f)\ge \langle w,h-f\rangle\qquad\forall h\in H^+
\end{equation}
 satisfies $Af\ge w$ and  $w^+\ge Af$, where $A$ is defined by~\eqref{defA}.  
\end{theo}
\begin{proof} The first inequality is obtained by choosing $h=f+\varphi$ with $\varphi$ arbitrary in $C^\infty_c(\Omega)^+$.  For the second inequality, we consider the solution $\sigma$ of the following variational problem:
\begin{equation}
\label{aux3.1}
\begin{cases}
\sigma\in H,\ \sigma\le f\\
a(\sigma,h-\sigma)\ge \langle w^+, h-\sigma\rangle\quad\forall h\in H\mbox{ such that }h\le f.
\end{cases}
\end{equation}
The existence of $\sigma$ is standard (see, for instance,~\cite{Kinderlehrer_Stampacchia}). We will prove that 
\begin{equation}
\label{aux3.2}
\sigma=f.
\end{equation}
Then, choosing $h=f-\varphi$ in~\eqref{aux3.1}, with $\varphi$ arbitrary in $C^\infty_c(\Omega)^+$, we find
$$a(f,-\varphi)\ge \langle w^+,-\varphi\rangle,$$
which is the second expected inequality. 

In order to prove~\eqref{aux3.2}, we first show that $\sigma\ge 0$ a.e. in $\Omega$.  Since $f\ge 0$ and $\sigma\le f$, we have $\sigma^+\le f$. We can therefore choose $h=\sigma^+$ in~\eqref{aux3.1}, and we obtain 
$$a(\sigma,\sigma^-)\ge \langle w^+,\sigma^-\rangle.$$
Since $a(\sigma^+,\sigma^-)= 0$, this implies
$$-a(\sigma^-,\sigma^-)\ge \langle w^+,\sigma^-\rangle\ge 0,$$
and so $\sigma^-=0$ by coercivity of $a$. 

Thus, $\sigma \ge 0$ a.e. in $\Omega$, and  we can choose $h=\sigma$ in~\eqref{aux3.0}. This yields
$$a(f,\sigma-f)\ge \langle w,\sigma-f\rangle,$$
and so 
$$a(f-\sigma,\sigma-f)\ge \langle w,\sigma-f\rangle+a(\sigma,f-\sigma)\ge \langle w^-,f-\sigma\rangle\ge 0,$$
where we used~\eqref{aux3.1} with $h=f$. By coercivity of $a$ again, we obtain~\eqref{aux3.2}. 
\end{proof}

We can now state our regularity result. The space $W^{2,p}(\Omega)$ is the $L^p(\Omega)$-Sobolev space~\cite{Brezis}, and $C^1(\overline{\Omega})$ denote the space of functions $f$ which are continuously differentiable in $\Omega$ and such that $f$ and $\nabla f$  are uniformly continuous in $\Omega$.  
\begin{theo}
\label{theoreg}
The solution $f^{\epsilon,v}$ of problem $\mathcal{P}^{\Lambda,\epsilon}$ belongs to $W^{2,p}(\Omega)$ for all $1\le p<\infty$. In particular, $f^{\epsilon,v}\in C^1(\overline{\Omega})$. 
\end{theo}
\begin{proof}By Proposition~\ref{propEL}, the solution $f\equiv f^{\epsilon,v}$ satisfies~\eqref{aux3.0} with $w$ defined by
$$w(x,z)=C-\int_\Omega k(x,z,x',z')f(x',z')dx'dz'.$$
In particular, $w$ belongs to $L^\infty(\Omega)$ with 
$$\|w\|_{L^\infty(\Omega)}\le |C|+\|k\|_{L^\infty(\Omega\times\Omega)}|\Omega|^{1/2}\|f\|_{L^2(\Omega)}<+\infty ,$$
since
$$\|k\|_{L^\infty(\Omega\times\Omega)}\le \frac{4\rho g v^3}{\pi}\int_1^\Lambda \lambda^2 d\mu(\lambda)<+\infty.$$
Thus, by Theorem~\ref{theoreg1}, $Af=\tilde{w}$ with $w\le \tilde{w}\le w^+$, so $\tilde{w}\in L^\infty(\Omega)$. 
We can use the regularity of the Laplacian on a rectangle with  Dirichlet boundary condition on three sides and  Neumann boundary condition on one side (Lemma~4.4.3.1 and Theorem~4.4.3.7 in~\cite{Grisvard}): we conclude that the solution $f\in H$ of $Af=\tilde{w}$ belongs to $W^{2,p}(\Omega)$ for all $1\le p<\infty$. For $p$ large enough, we have the Sobolev injection $W^{2,p}(\Omega)\subset C^1(\overline{\Omega})$~\cite{Adams}, and this concludes the proof. 
\end{proof}
\begin{rem}
The global regularity result obtained in Theorem~\ref{theoreg} is optimal because the domain $\Omega$ is a rectangle, so that even for the unconstrained problem, we do not expect a better global regularity in general~\cite{Grisvard}. However, in the open set $\{f^{\epsilon,v}>0\}$, the function $f^{\epsilon,v}$ is obviously $C^\infty$, by a classical bootstrap argument~\cite{Brezis}; otherwise, $f^{\epsilon,v}$ has the $C^{1,1}_{loc}(\Omega)$ regularity which is optimal for obstacle-type problems~\cite{Petrosyan_etal}. 
\end{rem}

\subsection{Three remarks on the limit case $\epsilon=0$}
In this section, for the reader's convenience,  we recall three results from~\cite[chapter 6]{Kostyukov}, which are related to our minimization problem in the limiting case $\epsilon=0$. The first two results are due to Krein. 

We first have:
\begin{prop}
\label{prop5.1}
If the wave resistance $R^\Lambda$ is computed by the integral~\eqref{RLambda1}, then for all $v>0$ and for all $f\in C_V$, $R^\Lambda(v,f)>0$.  
\end{prop}
\begin{proof}Let $v>0$, $f\in C_V$ and  assume by contradiction that $R^\Lambda(v,f)=0$. Then by~\eqref{RLambda1}, $T_f(v,\lambda)=0$ for every $\lambda\in[1,\Lambda]$, and by analycity, $T_f(v,\lambda)=0$ for all $\lambda\in\Rr$. Integrating by parts with respect to $x$ and using $f(-L/2,z)=f(L/2,z)=0$,  we obtain: 
$$0=T_f(v,\lambda)=i\lambda v \int_{-L/2}^{L/2}\int_0^Tf(x,z)e^{-\lambda^2 vz}e^{-i\lambda v x}dxdz\quad (\lambda\in\Rr).$$
Next, we use that the Fourier transform of a Gaussian density is known: 
$$\int_\Rr e^{-\lambda^2 v z'}e^{-i\lambda v x}=\sqrt{\frac{\pi}{vz'}}e^{-vx^2/(4z')}\quad (z'>0).$$ 
We multiply $T_f(v,\lambda)$ by $e^{-\lambda^2 a}$ with $a>0$ and we integrate on $\Rr$. By changing the order of integration (which is possible thanks to the new term), we find: 
\begin{eqnarray*}
0&=&\int_{-L/2}^{L/2}\int_0^Tf(x,z)\left(\int_\Rr e^{-\lambda^2 v(z+a)}e^{-i\lambda v x}d\lambda\right)dxdz\\
&=&\int_{-L/2}^{L/2}\int_0^Tf(x,z)\sqrt{\frac{\pi}{v(a+z)}}e^{-vx^2/(4(a+z)}dxdz.
\end{eqnarray*}
This is possible only if $f$ changes sign, hence a contradiction. The result is proved. 
\end{proof}
As pointed out by Krein, in Proposition~\ref{prop5.1}, it is essential to assume that the ship has a finite length.   Indeed, there exists a ship of infinite length which has a zero wave resistance. 
More precisely, let $f(x,z)=g(x)h(z)$ with 
$$g(x)=\frac{2}{\pi}\frac{\sin^2(ax/2)}{ax^2} $$
for some $a>0$ and where $h(z)$ is arbitrary. Then we have
$$\int_\Rr g(x)e^{-i\lambda v x}dx=\begin{cases}
(1-|\lambda|v/a)&\mbox{ if }|\lambda|<a/v,\\
0&\mbox{ if }|\lambda|\ge a/v. 
\end{cases}$$
On the other hand, integrating by parts with respect to $x$ in the definition of $T_f$ yields
$$T_f(v,\lambda)=i\lambda v\left(\int_{\Rr}g(x)e^{-i\lambda v x}dx\right)\left(\int_{\Rr_+}h(z)e^{-\lambda^2 vz}dz\right).$$
Thus, choosing $a<v$ yields $R^\Lambda(v,f)=0$ when $R^\Lambda$ is defined by~\eqref{defRLambda}. Such a choice of $g$ can be thought of as an endless caravan of ships.

Proposition~\ref{prop5.1} requires that $f\ge 0$ on $\Omega$. If we relax this assumption, for every $v>0$, it is possible~\cite{Kostyukov} to find $f\in C^\infty_c(\Omega)$  such that $T_f(v,\lambda)=0$ for all $\lambda$. Indeed, let $h\in C^\infty_c(\Omega)$ and set $f=\partial^2_xh+v\partial_zh$. Using several integration by parts and the identity 
$$(\partial^2_x-v\partial_z)\left(e^{-\lambda^2 vz}e^{-i\lambda v x}\right)=0,$$
we obtain
\begin{eqnarray}
\label{aux3.20}
T_f(v,\lambda)&=&i\lambda v \int_{-L/2}^{L/2}\int_0^Tf(x,z)e^{-\lambda^2 vz}e^{-i\lambda v x}dxdz=0.
\end{eqnarray}
This shows that the operator  $f\mapsto T_f(v,\cdot)$ is far from being one-to-one, as confirmed by the numerical simulations (see Section~\ref{sec5.1.2}).    

\section{Numerical methods}
\label{sec4}
In this section, we focus on the discretization of the minimization problem. Recall that the regularized criterion reads
$$N^{\Lambda,\epsilon}(v,f)= R^\Lambda(v,f)+\epsilon\int_\Omega|\nabla f(x,z)|^2dxdz,$$
where $\Lambda$ is taken large enough. The set of constraints will insure the fact that:
\begin{itemize}
\item[\textbullet] the volume of the (immerged) hull is given: $$\int_{\Rr^2} f(x,z) dxdz = V;$$
\item[\textbullet] the hull does not cross the center plane: $f(x,z) \geq 0$; 
\item[\textbullet] the hull is contained in a finite domain given by a box $\Omega=[-L/2,L/2] \times [0,T]$, where: $f(-L/2,\cdot)=f(L/2,\cdot)=f(\cdot,T)=0$.
\end{itemize}
The first constraint is an important one, since if no volume was imposed for the hull, the optimal solution to our problem would be $f=0$, for all target velocities $v$.

\subsection{A $Q^1$ finite element discretization}
\label{subsec4.1}
We adopt here a finite element approach in the sense that the optimal shape $f$ will be sought in a finite dimensional subspace 
$$V^h \subset H\subset H^1(\Omega).$$ 
We use a cartesian grid which divides the domain $\Omega=(-L/2,L/2)\times (0,T)$  into $N_x\times N_z$ small rectangles of size $\delta x\times \delta z$, where $\delta x=L/N_x$ and $\delta z=T/N_z$. We choose to represent the surface with the help of $Q^1$ finite-element functions: for every node $(x_i,z_i)$ of the grid, we define the ``hat-function''
\begin{align}
& e_i(x,z)=\dfrac{(x-(x_i- \delta x)) (z-(z_i- \delta z))}{\delta x \delta z} \, ,  \quad \text{for } (x,z) \in [x_i- \delta x , x_i] \times [z_i -\delta z, z_i]  \, , \notag \\
& e_i(x,z)=\dfrac{((x_i+ \delta x)-x) ((z_i+ \delta z)-z)}{\delta x \delta z} \, ,  \quad \text{for }  (x,z) \in [x_i , x_i+\delta x] \times [z_i, z_i +\delta z]  \, , \notag \\
& e_i(x,z)=\dfrac{(x-(x_i- \delta x)) ((z_i+ \delta z)-z)}{\delta x \delta z} \, ,  \quad  \text{for } (x,z) \in [x_i - \delta x , x_i] \times [z_i, z_i +\delta z]  \, , \notag \\
& e_i(x,z)=\dfrac{((x_i+ \delta x)-x) (z-(z_i- \delta z))}{\delta x \delta z} \, ,  \quad  \text{for } (x,z) \in [x_i , x_i + \delta x] \times [z_i - \delta z  , z_i]  \, , \notag \\
&e_i(x,z)=0,\quad\text{otherwise}. 
\label{defei}
\end{align}
Let us denote :
\begin{align}
& \mathcal{X}_i^+= [x_i , x_i+\delta x]  \, ,  \\
& \mathcal{X}_i^-= [x_i - \delta z , x_i]  \, ,  \\
& \mathcal{Z}_i^+=  [z_i, z_i +\delta z] \, ,  \\
& \mathcal{Z}_i^-=  [z_i - \delta z, z_i] \, .  
\end{align}
We can recast $e_i$ in the following manner, which is useful for further calculations:
\begin{equation}
e_i(x,z)= \dfrac{1}{\delta x \, \delta z}  a_i(x) \, b_i(z)  \, , \label{base}
\end{equation}
where:
\begin{align}
& a_i (x) =  ((x_i+ \delta x)-x) \mathbbm{1}_{\mathcal{X}_i^+}(x) +  (x-(x_i- \delta x)) \mathbbm{1}_{\mathcal{X}_i^-}(x)  \, ,  \label{base1}\\ 
& b_i (z) = ((z_i+ \delta z)-z) \mathbbm{1}_{\mathcal{Z}_i^+}(z) +  (z-(z_i- \delta z)) \mathbbm{1}_{\mathcal{Z}_i^-}(z)  \, ,  \label{base2}
\end{align}
where $\mathbbm{1}_A$ is the indicator function of the set $A$ (which is one in $A$ and zero outside of $A$). \par

In order to set $f(-L/2,\cdot)=f(L/2,\cdot)=f(\cdot,T)=0$ once and for all, we only keep the hat-functions which correspond to interior nodes or to nodes $(x_i,z_i)$ such that  $z_i=0$, $x_i\in(-L/2,L/2)$ (i.e. nodes on the upper side of $\Omega$).  These hat-functions are indexed from $1$ to $N_{int}$ (with $N_{int}=(N_x-1) (N_z-1)$) for  the interior nodes and from $N_{int}+1$ to $N=N_{int}+N_x-1$ for the $N_x-1$ nodes of the upper side. 

The functions  $\{e_i(x,z)\}_{i=1...N}$ are a basis of $V^h$, so that the hull surface is  represented by:
\begin{equation}
  \label{num-1}
f(x,z)=\somme{i=1}{N}{f_i e_i(x,z)},
\end{equation}
This identifies the space $V^h$  to $\Rr^N$, and in all the following we will denote $F=(f_i)_{i=1..N}$ the (column) vector in $\Rr^N$ corresponding to $f(x,z)$. 

 The  other two constraints described earlier read:
\begin{itemize}
\item[\textbullet] the volume of the hull is given: 
$$\sum_{i=1}^{N_{int}}{f_i}+\frac{1}{2}\sum_{i=N_{int}+1}^{N}f_i= \tilde{V},$$
where $\tilde{V}=V/(\delta x\delta z)$;
\item[\textbullet] the hull does not cross the center plane: $f_i \geq 0$ for $i=1\ldots N$.
\end{itemize}
Remark that, from a  geometrical point of view, this set of constraints can be seen as a (N-1)-dimensional simplex.
\subsection{Approximation of the wave resistance}
First, let us recall the expression of Michell's wave resistance as a function of the hull shape. Since the optimal ship has to be symmetric with respect to $x$ (see Theorem~\ref{theowellposed}), we drop the antisymmetric contribution $I$ of the hull on the wave resistance:  
$$R_{Michell}=\frac{4\rho g^2}{\pi U^2}\int_1^\Lambda J(\lambda)^2 \frac{\lambda^2}{\sqrt{\lambda^2-1}}d\lambda,$$
with
\begin{equation}
\label{Jbis}
J(\lambda)=\int_{-L/2}^{L/2}\int_0^T\frac{\partial f(x,z)}{\partial x}\exp\left(-\frac{\lambda^2 gz}{U^2}\right)\sin\left(\frac{\lambda g x}{U^2}\right)dxdz.
\end{equation}
Integrating by parts in~\eqref{Jbis}, and denoting $v=g/U^2$, we obtain the simpler expression
$$R_{Michell} = \dfrac{4 \rho g v^3}{\pi} \integ{1}{\Lambda}{ \tilde{J}^2(\lambda) \dfrac{\lambda^4}{\sqrt{\lambda^2-1}} \mathrm{d} \lambda} \, ,$$
with
\begin{equation}
\tilde{J}(\lambda) = \integ{\mathbb{R}\times \mathbb{R}^+}{}{f(x,z) \, e^{- \lambda^2 v z} \, \cos(\lambda v x) \, \mathrm{d}x \mathrm{d}z} \, , 
\end{equation}
Since $R_{Michell}$ is a quadratic form with respect to $f$, when $f$ is given as~\eqref{num-1}, the expression of the wave  resistance reads
\begin{equation}
\label{RMh}
R_{Michell}= \dfrac{4 \rho g v^3}{\pi} F ^t \, M_w \, F   \, ,
\end{equation}
where $F^t$ denotes the transpose of the vector $F$. 
Simple calculations give us the $N \times N$ matrix $M_w$:
\begin{equation}
M_w = \integ{1}{\Lambda }{ \mathcal{J}(\lambda)\mathcal{J}(\lambda)^t \, \dfrac{\lambda^4}{\sqrt{\lambda^2-1}} \mathrm{d} \lambda} \, , \label{num-2}
\end{equation}
where  $\mathcal{J}(\lambda)$ is the (column) vector of $\Rr^N$ given by
\begin{equation}
(\mathcal{J}(\lambda))_i =  \integ{\mathbb{R}\times \mathbb{R}^+}{}{e^{- \lambda^2 v z} \, \cos(\lambda v x) \, e_i(x,z)  \, \mathrm{d}x \mathrm{d}z } \, , \label{num-2pp}
\end{equation}
for $i=1\ldots N$. Every basis function $e_i$ is the product of a polynomial in $x$ by a polynomial in $z$ on every one of the cells (see~\eqref{defei}), so one can compute exactly the values of $\mathcal{J}(\lambda)$. Injecting~\eqref{base} into~\eqref{num-2pp}, we obtain :
\begin{equation}
(\mathcal{J}(\lambda))_i =  \integ{\mathbb{R}\times \mathbb{R}^+}{}{e^{- \lambda^2 v z} \, \cos(\lambda v x) \, a_i(x) \, b_i(z)  \, \mathrm{d}x \mathrm{d}z } \, .
\end{equation}
Hence our integral can be written as a product of two independent integrals:
\begin{equation}
(\mathcal{J}(\lambda))_i =  \integ{\mathbb{R}}{}{ \cos(\lambda v x) \, a_i(x)  \, \mathrm{d}x}  \integ{ \mathbb{R}^+}{}{e^{- \lambda^2 v z} \, b_i(z)  \, \mathrm{d}z } \, .
\end{equation}
From~\eqref{base1} and~\eqref{base2}, we remark that each integral is the sum of two terms:
\begin{align}
&\integ{\mathbb{R}}{}{ \cos(\lambda v x) \, a_i(x)  \, \mathrm{d}x} = a_i^+ + a_i^- \, ,  \\
&\integ{ \mathbb{R}^+}{}{e^{- \lambda^2 v z} \, b_i(z)  \, \mathrm{d}z } =  b_i^+ + b_i^- \, , 
\end{align}
where:
\begin{align}
& a^+_i = \integ{x_i}{x_i+\delta x}{ \cos(\lambda v x) \,  ((x_i+ \delta x)- x)  \, \mathrm{d}x} \, , \\
& b^+_i = \integ{ z_i}{z_i+\delta z}{e^{- \lambda^2 v z} \, ((z_i+ \delta z)- z) \, \mathrm{d}z }\, , \\
& a^-_i =  \integ{x_i-\delta x}{x_i}{ \cos(\lambda v x) \,  (x-(x_i- \delta x))  \, \mathrm{d}x} \, , \\
& b^-_i = \integ{ z_i- \delta z}{z_i}{e^{- \lambda^2 v z} \, (z-(z_i- \delta z)) \, \mathrm{d}z } \, .
\end{align}
Hence our vector $\mathcal{J}(\lambda)$ writes:
\begin{equation}
(\mathcal{J}(\lambda))_i = \dfrac{1}{\delta x \, \delta z}(a^+_i + a^-_i)(b^+_i +b^-_i) \, .
\end{equation}
Elementary yet tedious calculations give us the values for the integrals $a_i^+$, $b_i^+$, $a_i^-$ and $b_i^-$:
\begin{align}
& a^+_i =\frac{1}{v^2 \lambda^3}\left\{ - \delta x \, \sin(\lambda v x_i) + \frac{1}{\lambda v}(\cos(\lambda v x_i)-\cos(\lambda v (x_i+\delta x) ))\right\} \; ,\label{aip} \\
& b^+_i =\frac{1}{v^2 \lambda^3}\left\{ \delta z \, e^{- \lambda^2 v z_i} - \frac{1}{\lambda^2 v}(e^{-\lambda^2 v z_i}-e^{-\lambda^2 v( z_i + \delta z)}) \right\}\; , \\
& a^-_i = \frac{1}{v^2 \lambda^3}\left\{ \delta x \, \sin(\lambda v x_i) + \frac{1}{\lambda v}(\cos(\lambda v x_i)-\cos(\lambda v (x_i-\delta x) ))\right\} \; , \\
& b^-_i = \frac{1}{v^2 \lambda^3}\left\{- \delta z \, e^{- \lambda^2 v z_i} - \frac{1}{\lambda^2 v}(e^{- \lambda^2 v z_i}-e^{- \lambda^2 v( z_i - \delta z)}) \right\}\, . \label{bim}
\end{align}
Moreover, $b_i^-=0$ if $z_i=0$. 

Let us now describe the method employed to approximate the integral with respect to $\lambda$ which appears in~\eqref{num-2}. In~\cite{Tarafder}, Tarafder \textit{et. al.} described an efficient method in order to compute this integral. In order to get rid of the singular term for $\lambda=1$, the integral is transformed in the following manner:
\begin{align}
M_w &= \integ{1}{\Lambda }{ \mathcal{J}(\lambda)\mathcal{J}(\lambda)^t \, \dfrac{\lambda^4}{\sqrt{\lambda^2-1}} \mathrm{d} \lambda} \\
&= \integ{1}{2}{ \mathcal{J}(\lambda)\mathcal{J}(\lambda)^t \, \dfrac{\lambda^4}{\sqrt{\lambda^2-1}} \mathrm{d} \lambda} + \integ{2}{\Lambda}{ \mathcal{J}(\lambda)\mathcal{J}(\lambda)^t \, \dfrac{\lambda^4}{\sqrt{\lambda^2-1}} \mathrm{d} \lambda} \\
&= \mathcal{J}(1)\mathcal{J}(1)^t \integ{1}{2}{\dfrac{1}{\sqrt{\lambda^2-1}}}\mathrm{d}\lambda +  \integ{1}{2}{ \dfrac{ \lambda^4 \, \mathcal{J}(\lambda)\mathcal{J}(\lambda)^t -\mathcal{J}(1)\mathcal{J}(1)^t}{\sqrt{\lambda^2-1}} \mathrm{d} \lambda} \nonumber\\
& \hspace{6cm}+ \integ{2}{\Lambda}{ \mathcal{J}(\lambda)\mathcal{J}(\lambda)^t \, \dfrac{\lambda^4}{\sqrt{\lambda^2-1}} \mathrm{d} \lambda}
\end{align}
The first integral can be computed explicitly:
\begin{equation}
\mathcal{J}(1)\mathcal{J}(1)^t \integ{1}{2}{\dfrac{1}{\sqrt{\lambda^2-1}}}\mathrm{d}\lambda= \ln(2+\sqrt{3}) \, \mathcal{J}(1)\mathcal{J}(1)^t \, . \label{num-3}
\end{equation}
The second integral, which is not singular anymore, is computed with a second order midpoint approximation formula:
\begin{equation}
\integ{1}{2}{ \dfrac{ \lambda^4 \, \mathcal{J}(\lambda)\mathcal{J}(\lambda)^t -\mathcal{J}(1)\mathcal{J}(1)^t}{\sqrt{\lambda^2-1}} \mathrm{d} \lambda} \approx \somme{i=1}{N_0}{\dfrac{ \lambda_{i,0}^4 \, \mathcal{J}(\lambda_{i,0})\mathcal{J}(\lambda_{i,0})^t -\mathcal{J}(1)\mathcal{J}(1)^t}{\sqrt{\lambda_{i,0}^2-1}} \delta \lambda_0} \, , \label{num-4}
\end{equation}
where $\delta \lambda_0=1/N_0$ and $\lambda_{i,0}=1+ \left( i+ \frac{1}{2} \right) \delta \lambda_0$ (for $i=1$,\ldots,$N_1$). Thanks to the exponential decay of $\mathcal{J}(\lambda)$ when $z_i>0$ and $z_i-\delta z>0$ (see~\eqref{aip}-\eqref{bim}),  the function under the third integral has an exponential decay for most values of $z$. Therefore,  the third integral is cut in intervals of exponentially growing lengths (we set $\Lambda=2^{K_{\Lambda}}$):
\begin{equation}
 \integ{2}{\Lambda}{ \mathcal{J}(\lambda)\mathcal{J}(\lambda)^t \, \dfrac{\lambda^4}{\sqrt{\lambda^2-1}} \mathrm{d} \lambda} = \somme{k=1}{K_{\Lambda}-1}{ \integ{2^k}{2^{k+1}}{ \mathcal{J}(\lambda)\mathcal{J}(\lambda)^t \, \dfrac{\lambda^4}{\sqrt{\lambda^2-1}} \mathrm{d} \lambda}}
\end{equation}
On each interval, the integral is computed with a second order midpoint approximation formula:
\begin{equation}
\integ{2^k}{2^{k+1}}{ \mathcal{J}(\lambda)\mathcal{J}(\lambda)^t \, \dfrac{\lambda^4}{\sqrt{\lambda^2-1}} \mathrm{d} \lambda} \approx  \somme{i=1}{N_{k}}{\mathcal{J}(\lambda_{i,k})\mathcal{J}(\lambda_{i,k})^t \dfrac{\lambda_{i,k}^4}{\sqrt{\lambda_{i,k}^2-1}} \delta \lambda_k} \, , \label{num-5}
\end{equation}
where: $\delta \lambda_k=\dfrac{2^k}{N_{k}}$, and $\lambda_{i,k}=2^k+(i+\frac{1}{2}) \delta \lambda_k$ for $i=1,\ldots,N_k$.

\begin{rem}The integration method with respect to $\lambda$ described above preserves the positivity of the operator $M_w$. From~\eqref{num-3}, ~\eqref{num-4} and~\eqref{num-5}, the approximation of $M_w$ can be written as:
\begin{equation}
M_w= \omega_0 \mathcal{J}(1)\mathcal{J}(1)^t  + \somme{j=1}{J^\star}{\omega_j \mathcal{J}(\lambda_j)\mathcal{J}(\lambda_j)^t } \label{num-6}
\end{equation}
where the sequence $(\lambda_j)$ contains all the midpoints $\lambda_{i,k}$ described above. It is clear by construction that $\omega_j > 0$ for $j\ge 1$. For $\omega_0$, the matter is less obvious, and the positivity is a consequence of the choice we made for the numerical method of integration. The coefficient $\omega_0$ reads
\begin{equation}
\omega_0=  \ln(2+\sqrt{3}) -  \somme{i=1}{N_0}{\dfrac{\delta \lambda_0}{\sqrt{\lambda_{i,0}^2-1}}} 
\end{equation}
The first term of this difference is the exact integral, and the second term is the approximate integral. When we deal with the integral of convex functions, the approximate integral computed with the  midpoint approximation is always lower than the exact integral. Since $\lambda \rightarrow \dfrac{1}{\sqrt{\lambda^2-1}}$ is convex for $\lambda>1$, we have $\omega_0 >0$. Hence, $M_w$ is positive (semi-definite, see Figure~\ref{valp}).
\end{rem}

\subsection{Approximation of the viscous resistance}
Let us give the expression of the additional ``viscous drag" term when $f$ is given by~\eqref{num-1}:
\begin{align}
\integ{\Omega}{}{|\nabla f |^2}{} & = ( \nabla f \; , \; \nabla f)_{L^2} \\
& = ( \somme{i=1}{N}{f_i \nabla e_i} \; , \;  \somme{i=1}{N}{f_i \nabla e_i})_{L^2} \\
& = \somme{i,j=1}{N}{f_i \, f_j ( \nabla e_i \; ,  \nabla e_j)_{L^2}}\label{Md}
\end{align}
The  computation of the matrix $M_d= (\nabla e_i \; ,  \nabla e_j)_{L^2}$ is standard~\cite{Lucquin}. This matrix is nondiagonal, symmetric positive definite.

\subsection{Method of optimization}

From~\eqref{RMh}  and~\eqref{Md}, we can recast the optimization problem  as finding $F^*$ which solves
\begin{equation}
F^*=\underset{F \in K_{\tilde{V}}}{\text{argmin}}\left\{ F^t\,  (\frac{4\rho g v^3}{\pi} M_w+\epsilon M_d) F \right\} \, ,
\end{equation}
where
\begin{equation}
K_{\tilde{V}}=\left\{ F=(f_i)_{1\le i\le N}\in \mathbb{R}^N \; : \; F \geq 0   \quad \text{and} \quad \sum_{i=1}^{N_{int}}{f_i}+\frac{1}{2}\sum_{i=N_{int}+1}^Nf_i =\tilde{V} \right\}.
\end{equation}
This problem can be reformulated as finding the saddle point $(F^*,\lambda_1^*,\lambda_2^*) \in  \mathbb{R}^N  \times (\mathbb{R}_-)^N \times \mathbb{R}$ for the following Lagrangian:
\begin{equation}
\mathcal{L}(F,\lambda_1,\lambda_2)= F^t \,  (\frac{4\rho g v^3}{\pi} M_w+\epsilon M_d) F +  \lambda_1^t F + \left(\sum_{i=1}^{N}\alpha_i{f_i} - \tilde{V} \right) \lambda_2,
\end{equation}
where $\alpha_i=1$ for $i\in\{1,\ldots,N_{int}\}$ and $\alpha_i=1/2$ otherwise. 

The method we used in order to find this saddle point is the Uzawa algorithm~\cite{Ciarlet}. Given $(F^n,\lambda_1^n,\lambda_2^n)$, we find $(F^{n+1},\lambda_1^{n+1},\lambda_2^{n+1})$ in the following manner:
\begin{itemize}
\item First, we obtain $F^{n+1}$ by minimizing $\mathcal{L}(F,\lambda_1^n,\lambda_2^n)$ with respect to $F$ in $\mathbb{R}^N$, which is equivalent to: 
\begin{equation}
F^{n+1}= ((\frac{4\rho g v^3}{\pi}M_w+\epsilon M_d)^{-1} (\lambda_1 + \lambda_2) \, , 
\end{equation}
\item then we iterate on the Lagrange multipliers with:
\begin{align}
\lambda_1^{n+1} & =\mathbb{P}_{(\mathbb{R}_-)^N}\left( \lambda_1^{n} + \delta r_1\, F \right) \,  , \\
\lambda_2^{n+1} & = \lambda_2^{n} + \delta r_2  \left(\sum_{i=1}^{N}\alpha_i{f_i} - \tilde{V} \right) \, ,
\end{align}
where $\mathbb{P}_{(\mathbb{R}_-)^N}$ denotes the projection on $(\mathbb{R}_-)^N$, $ \delta r_1$ and $ \delta r_2$ are steps that have to be taken small enough in order to insure convergence, and large enough in order to insure fast convergence.
\end{itemize}
When this algorithm has converged (\textit{i.e.} $(F^{n+1}-F^n,\lambda_1^{n+1}-\lambda_1^n,\lambda_2^{n+1}-\lambda_2^n)$ small enough for some norm), the saddle point is reached.

\section{Numerical results and their interpretation}
\label{sec5}
In this section, we  perform hull optimization with the method described above.  We first describe the necessity of adding a coercive term in our optimization criterion, and then we give some optimized hulls obtained for moderate Froude numbers. 

We used the following set of parameters, which could correspond  to an experiment in a towing basin:  $\rho=1000\,\mathrm{kg}\cdot \mathrm{m}^{-3}$,  $g=9.81\,\mathrm{m}\cdot\mathrm{s}^{-2}$, $L=2\, \mathrm{m}$, $T=20\, \mathrm{cm}$, $V= 0.03\,  \mathrm{m}^3$. 

The space discretization parameters are $N_x=100$ and $N_z=20$ (except in Figure~\ref{valp} where $N_x=100$ and $N_z=30$). These values are taken as a compromise between the computational cost and the accuracy we seek. We remark that since $M_w$ is obtained as the product of two vectors with $N_x \times N_z$ entries, this matrix is a full matrix with $(N_x \times N_z)^2$ non-zero entries. this means that the memory cost is $O((N_x \times N_z)^2)$ (instead of $O(N_x \times N_z)$ for a sparse problem). \par
We remind that $\tilde{V}=V/(\delta x\delta z)$ with $\delta x=L/N_x$ and $\delta z=T/N_z$. The parameters $N_0$, \ldots, $N_{K_\Lambda}$  used in the numerical integration (see~\eqref{num-4}-\eqref{num-5}) are all equal to $80$. The integer $K_\Lambda$ is determined by a stopping criterion  ($K_\Lambda$ is generally around $10$).

The  velocity is given by the length Froude number:
\begin{equation}
\label{defFroude}
Fr=\dfrac{U}{\sqrt{gL}} \, ,
\end{equation} 
and we remind that in our notations, $v=g/U^2$. Our $Q^1$ discretized wave resistance formula~\eqref{RMh} has been validated by comparison to some tabulated results obtained by Kirsch~\cite{Kirsch}   for a hull of longitudinal parabolic shape with a rectangular cross-section (rectangular Wigley hull). We used the \texttt{Scilab} software for the computations and the \texttt{Matlab}\footnote{\textsf{http://www.mathworks.fr/}} software for the figures. 
\subsection{Degenerate nature of the wave resistance criterion for optimization}
\subsubsection{Letting $\epsilon$ tend to $0$}
Let us examine the numerical results of the optimization problem:
\begin{equation}
F^*=\underset{F \in K_{\tilde{V}}}{\text{argmin}}\left\{ F^t\, (\frac{4\rho g v^3}{\pi}M_w+\epsilon M_d) F \right\} \, , \label{eqres-1}
\end{equation}
for smaller and smaller values of $\epsilon$, with $Fr=1$.  In Figure~\ref{res-1} we notice that, as $\epsilon$ gets small ($\epsilon$ is expressed in $\textrm{Pa}$), the optimized hull does not seem to converge towards a limit. In fact most of the hull's volume   tends to accumulate on the edges of the domain boundaries, where $f=0$ is imposed. 
\begin{center}
\begin{figure}[!ht]
\includegraphics[width=10cm]{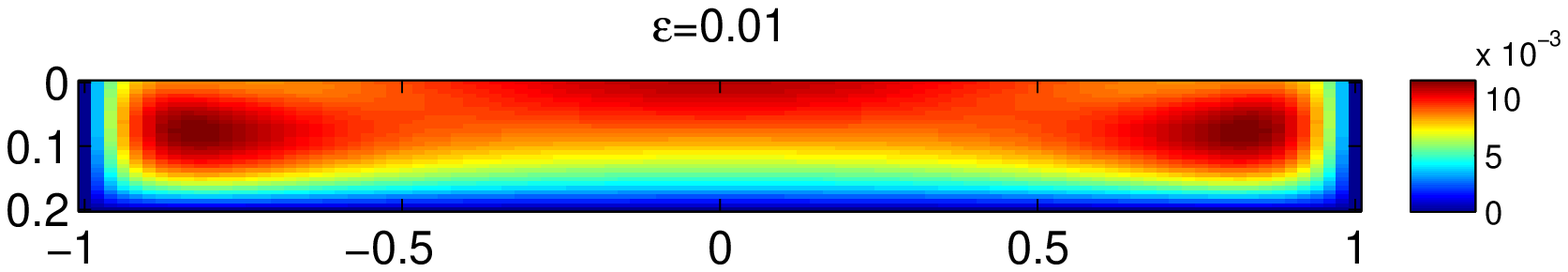}\\
\includegraphics[width=10cm]{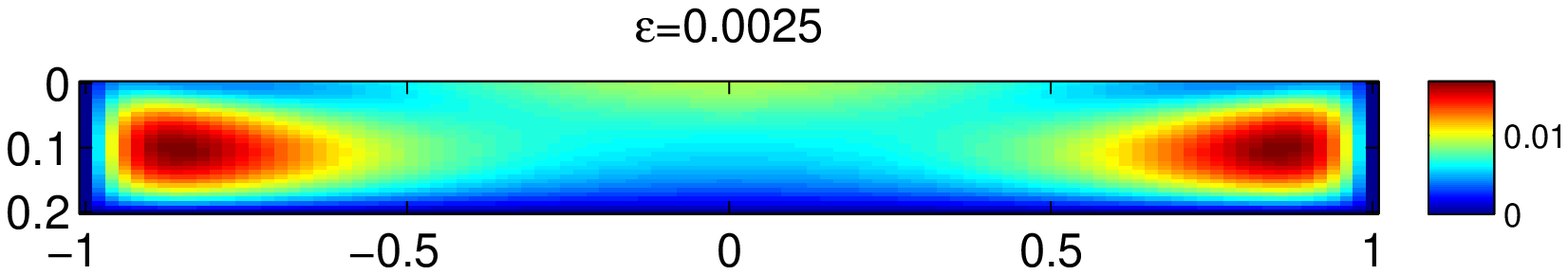} \\
\includegraphics[width=10cm]{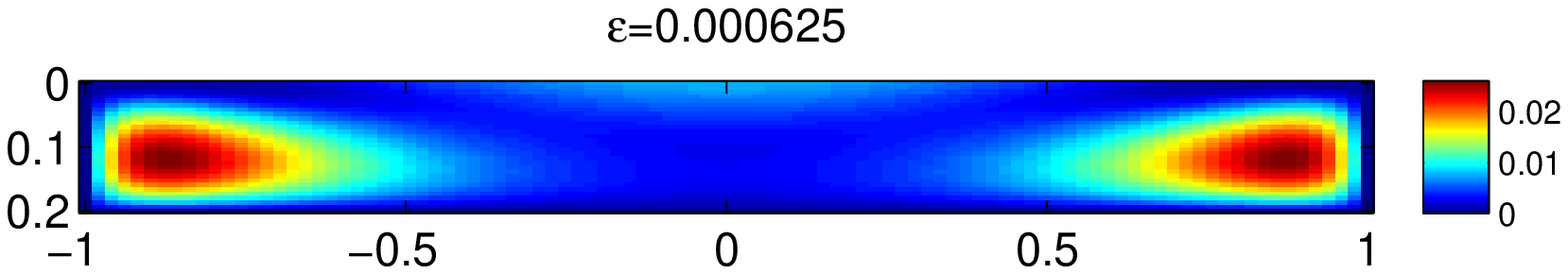} \\
\includegraphics[width=10cm]{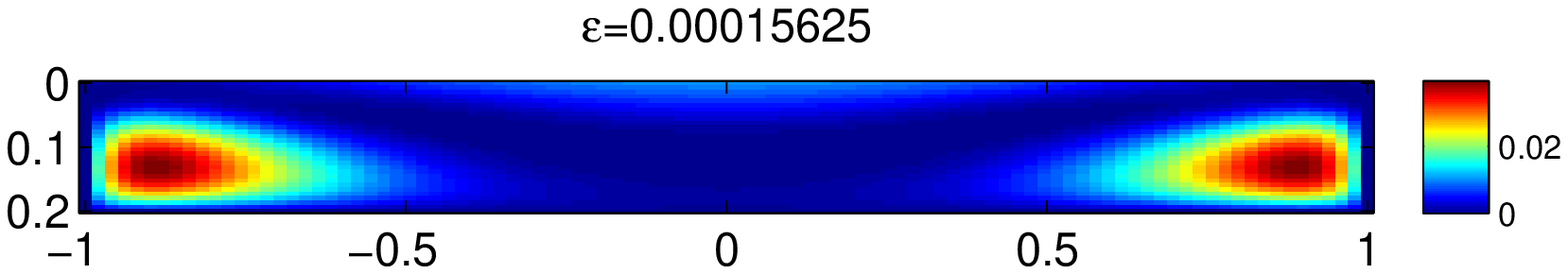} 
\caption{Color maps of the optimized hull function $f(x,z)$ for smaller and smaller values of $\epsilon$. \label{res-1}}
\end{figure}
\end{center}
Note that this phenomenon is very similar to a boundary layer phenomenon. Let us take the characteristic width of the boundary layer as the distance between the left border of the domain and the center of mass $(\bar{x},\bar{y})$ of the half hull. 
\begin{center}
\begin{figure}[!ht]
\includegraphics[width=8cm]{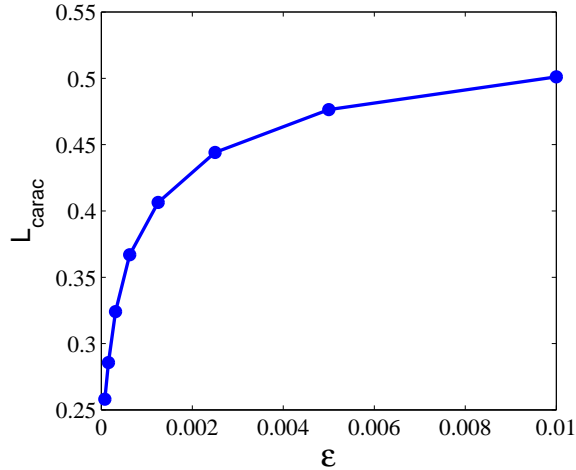}\\
\caption{Characteristic width of the boundary layer as a function of $\epsilon$. \label{res-2}}
\end{figure}
\end{center}
The characteristic width of the boundary layer (see Figure~\ref{res-2}) seems to fit a law of the type
\begin{equation}
L_{carac}(\epsilon) \sim \epsilon^{0.15} \, .
\end{equation}
This phenomenon suggests that the optimization problem $\mathcal{P}^{\Lambda,\epsilon}$  is ill-posed when $\epsilon=0$. 
\begin{rem} 
\label{rem5.1}
 In finite dimension, the problem 
$$F^\star\in \mbox{argmin}_{F\in K_{\tilde{V}}}F^t\, M_w\,F$$
has at least one solution, because $K_{\tilde{V}}$, being a simplex,  is a compact subset of $\Rr^N$.  The existence of such a solution  is due to the discretization. 
\end{rem}
\subsubsection{About the eigenvalues of  $M_w$ (numerics)}
\label{sec5.1.2}
Let us consider once again (for a fixed $v>0$) the operator $f\mapsto T_f(v,\cdot)$ (see~\eqref{defT}) which appears in the definition of $R^{\Lambda}(v,f)$~\eqref{defRLambda}.  We have seen that $T_f$ is not invertible (cf.~\eqref{aux3.20}). This (linear) operator transforms a  function of two variables, $f(x,z)$,  into a function of one variable, $\lambda$. Roughly speaking,  we ``loose'' one dimension in the process, and  this is the reason why the wave resistance $R^\Lambda$ alone is not suited for minimization.

This  is confirmed by numerical computation of the eigenvalues of the matrix $M_w$ (see Figure~\ref{valp}, where the Froude number is equal to $1$). Since $M_w$ is symmetric, up to a change of orthonormal basis, $M_w$ is equal to a diagonal matrix formed with its eigenvalues. Recall now that $M_w$ represents (up to a constant factor) the restriction of $R^\Lambda$ to the space $V^h$ (we omit here the fact that $M_w$ contains only the cosin term). In Figure~\ref{valp}, $N_x=100$ and $N_z=30$, so there are $N\approx 3000$ degrees of freedom,  but there are less than 200  positive eigenvalues (for an index $i\ge 200$, the eigenvalue satisfies $|\lambda_i|<10^{-15}$, which is the double precision accuracy; for $i\ge 1600$,  we have $\lambda_i=0$ up to computer accuracy, so that $\lambda_i$ is not represented in the logarithmic scale). Corollary~\ref{coro5.3} below provides  a theoretical lower bound ($N_x-1=99$) concerning the number of positive eigenvalues. 

In other words, Figure~\ref{valp} shows that only a few degrees of freedom are necessary in order to minimize efficiently the wave resistance. In such a case, existence of a solution  to the minimum wave resistance problem is a consequence of the discretization (see Remark~\ref{rem5.1}). This is an approach that has been used by many authors~\cite{Gotman,Guilloton,Higuchi_Maruo_1979,Hsiung_1981,Hsiung_Shenyyan_1984,Kostyukov,Michalski_etal,Pavlenko,Weinblum}. In contrast, with our approach, we do not need to impose ``a priori'' the set of parameters: the interesting degrees of freedom are selected when minimizing the total resistance. 

\subsubsection{About the eigenvalues of $M_w$ (analysis)}
\begin{center}
\begin{figure}[!ht]
\psfrag{i}{$i$}
\psfrag{e}{$|\lambda_i|$}
\includegraphics[width=8cm]{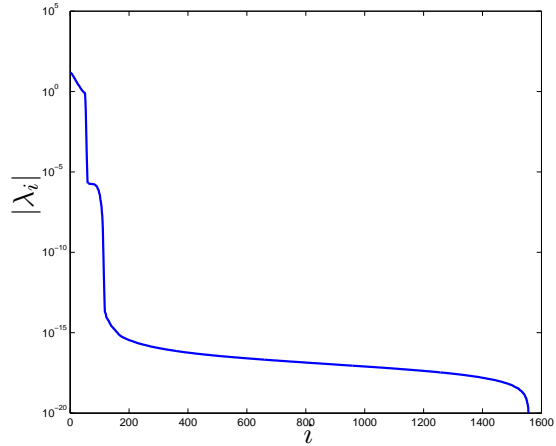}\\
\caption{\label{valp}Eigenvalues of $M_w$}
\end{figure}
\end{center}

Here, we provide a theoretical lower bound for the number of positive eigenvalues. 
  First, we notice that the operator $T_f$ can be seen as the composition of a Fourier transform in $x$ by a modified Laplace transform in $z$. More precisely, for $\varphi\in L^1(\Rr)$, let 
$$\mathcal{F}_x(\varphi)(\xi):=\int_\Rr\e^{-i\xi x}\varphi(x)dx\quad (\xi\in\Rr)$$
 be the Fourier transform of $\varphi$, and  for  $\chi\in L^1(\Rr_+)$,   let 
$$\mathcal{L}_z(\chi)(s):=\int_{\Rr_+}e^{-sz}\chi(z)$$
 be the Laplace transform of $\chi$, which is defined for all $s\in\Cc$ such that $\Re(s)\ge 0$. 
If $f(x,z)=\varphi(x)\chi(z)$ with $\varphi\in L^1(\Rr)$ and $\chi\in L^1(\Rr_+)$, then for all $v>0$, 
\begin{equation}
\label{remT}
T_f(v,\lambda)=\mathcal{F}_x(\varphi')(\lambda v)\mathcal{L}_z(\chi)(\lambda^2 v)\quad\forall \lambda\in\Rr.
\end{equation}
As a consequence, we have:
\begin{prop}
\label{propaux}
Assume that $f\in H$ can be written $f(x,z)=\varphi(x)\chi(z)$ with $\varphi\in H^1_0(-L/2,L/2)$ and $\chi\in H^1(0,T)$. If $f\not=0$, then for all $v>0$, the function $\lambda\to T_f(v,\lambda)$ is real  analytic  on $\Rr$ and not identically zero. 
\end{prop}
\begin{proof}Since $\varphi'\in L^1(-L/2,L/2)$, and since the kernel $(x,\xi)\mapsto \e^{-i\xi x}$ is holomorphic with respect to $\xi\in\Rr$ and uniformly bounded for $\xi$ in a compact subset of $\Cc$ and $x\in[-L/2,L/2]$, by standard results, the Fourier transform $\xi\mapsto\mathcal{F}_x(\varphi')(\xi)$ is holomorphic on $\Cc$ (where $\varphi'$ is extended by $0$ on $\Rr$).  The assumptions on $f$ and $\varphi$ imply that $\varphi'\not=0$; by injectivity of the Fourier transform on $L^1(\Rr)$, we have $\mathcal{F}_x(\varphi')\not=0$. Similarly, the Laplace transform $s\mapsto\mathcal{L}_z(\chi)(s)$ is holomorphic on $\Cc$. If $\chi\in H^1(0,T)$, then $\chi$ is absolutely continuous on $[0,T]$, and an inversion formula holds~\cite{Bellman_Cooke}. Thus, since  $\chi\not=0$ (by assumption), we have $\mathcal{L}_z(\chi)\not=0$. By analycity, $\mathcal{F}_x(\varphi')$ and $\mathcal{L}_z(\chi)$ have isolated roots. The conclusion follows from~\eqref{remT}. 
\end{proof}
When $R^\Lambda$ is defined by a numerical integration of the form~\eqref{RLambda2}, with nodes $1\le \lambda_1<\cdots<\lambda_{K^\star}\le \Lambda$, the maximum stepsize of the subdivision $(\lambda_k)$ is defined by
$$\delta \lambda_{max}=\max_{0\le k\le K^\star}(\lambda_{k+1}-\lambda_k),$$
where we have set $\lambda_0=1$ and $\lambda_{K^\star+1}=\Lambda$. 
Recall that $V^h$, introduced in Section~\ref{subsec4.1},  is the finite dimensional subspace of $H$ obtained by the conforming $Q^1$ discretization. Let $v>0$ be fixed. We can state:
\begin{coro}
\label{coro5.3} If $R^\Lambda$ is defined by the integral formula~\eqref{RLambda1}, or by a numerical integration~\eqref{RLambda2}  where the maximum stepsize is taken sufficiently small, there exists a subspace $W^h\subset V^h$ which has a  dimension greater than or equal to $ \max\{N_x,N_z\}-1$ and such that $R^\Lambda(v,f)>0$ for all $f\in W^h\setminus \{0\}$.   
\end{coro}
\begin{proof}We assume that $N_x\ge N_z$ (otherwise we exchange the roles of $x$ and $z$). We also assume (by changing the indexing if needed) that the hat-functions $e_1,\ldots,e_{N_x-1}$ are associated to   the first line of interior nodes $(x_i,z_1)$ with $x_i=-L/2+i\delta x$ ($i=1,\ldots,N_x-1$), $z_1=T-\delta z$. Every $e_i$ can be written 
\begin{equation}
\label{eigh}
e_i(x,z)=\varphi_i(x)\chi_1(z)
\end{equation}
where
$$\varphi_i(x)=\hat{\varphi}\left(\frac{x-x_i}{\delta x}\right),\quad \chi_1(z)=\hat{\varphi}\left(\frac{z-z_1}{\delta z}\right),\quad\hat{\varphi}(s)=\begin{cases}
1+s&\mbox{if }s\in[-1,0],\\
1-s&\mbox{if }s\in[0,1],\\
0&\mbox{otherwise}.
\end{cases}$$
Let $W^h$ be the subspace of $V^h$ generated by $\{e_1,\ldots,e_{N_x-1}\}$, and let $f\in W^h\setminus\{0\}$, i.e. $f(x,z)=\sum_{i=1}^{N_x-1}\alpha_ie_i(x,z)$ with $(\alpha_1,\ldots,\alpha_{N_x-1})\not=(0,\ldots,0)$. By~\eqref{eigh}, 
\begin{equation}
\label{fproduit}
f(x,z)=\left(\sum_{i=1}^{N_x-1}\alpha_i\varphi_i(x)\right)\chi_1(z)=\varphi(x)\chi_1(z),
\end{equation}
where $\varphi\in H^1_0(-L/2,L/2)$, $\chi_1\in H^1(0,T)$. 
Using Proposition~\ref{propaux}, we see that $\lambda\mapsto T_f(v,\lambda)$ is real analytic on $\Rr$ and not identically zero. Thus, if $R^\Lambda$ is defined by the integral formula~\eqref{RLambda1}, $R^\Lambda(v,f)>0$.

 Next, assume that  $R^\Lambda$ is defined by an numerical integration such as~\eqref{RLambda2}. We claim that if the maximum stepsize is sufficiently small, then $R^\Lambda(v,f)>0$ for all $f\in W^h\setminus\{0\}$. Otherwise, there exist a sequence of subdivisions $1\le \lambda_1^n<\cdots<\lambda_{K^n}^n\le \Lambda$ with maximum stepsize $\delta\lambda_{max}^n\to 0$ and $f^n=\sum_{i=1}^{N_x-1}\alpha_i^ne_i\in W^h\setminus\{0\}$ such that 
\begin{equation}
\label{eqaux10}
R^\Lambda(v,f^n)=0\iff T_{f^n}(v,\lambda_k^n)=0\ \forall k\in\{1,\ldots,K^n\}.
\end{equation}
Denote $\alpha^n=(\alpha_1^n,\ldots,\alpha_{N_x-1}^n)$, and 
$$\|\alpha^n\|_\infty=\max_{1\le i\le N_x-1}|\alpha_i^n|.$$
Replacing $\alpha^n$ by $\alpha^n/\|\alpha^n\|_\infty$ if necessary, we may assume  that $\|\alpha^n\|_\infty=1$.
  Thus, up to a subsequence, $\alpha^n\to \alpha$  in $\Rr^{N_x-1}$, with $\|\alpha\|_\infty=1$. The sequence of functions $f^n$ tends in $W^h$ to a function $f=\sum_{i=1}^{N_x-1}\alpha_ie_i\not=0$, which can be represented as in~\eqref{fproduit}. Using Proposition~\ref{propaux} again, we obtain that $\lambda\mapsto T_f(v,\lambda)$ is an analytic function with isolated zeros in $[1,\Lambda]$. On the other hand, passing to the limit in~\eqref{eqaux10} shows that $\lambda\mapsto T_f(v,\lambda)$ is identically equal to $0$ on $[1,\Lambda]$, yielding a contradiction. The claim is proved. 
\end{proof}

\subsection{Optimization with respect to the wave and viscous drag resistance}
In this section we examine the influence of the  velocity on the optimization problem~\eqref{eqres-1} for:
\begin{equation}
\epsilon = \frac{1}{2} \rho C_w U^2 \, ,
\end{equation}
with a fixed value for the effective viscous drag coefficient: $C_w= 10^{-2}$, which is a rather realistic value when considering a streamlined body. Note that all the results described below depend on the choice $C_w$, and the bounds of the different regimes described with respect to the Froude number may be affected if $C_w$ is changed.
When the Froude number (see~\eqref{defFroude}) is large, or when the Froude number is low (in our case $Fr\leq 0.1$ or $Fr\geq2$) we observe that the optimized shapes we obtain are very similar, and seem to essentially minimize the surface area of the hull (see Figure~\ref{res-3}). For large Froude numbers, the reason is that the wave resistance (which goes to $0$ as $Fr$ goes to infinity) is significantly smaller than the viscous resistance, and hence the optimal hull is close to the optimal hull for the viscous drag resistance, which depends mainly on the surface area and $Fr^2$. For low Froude numbers, the reason is not so clear, but in this case, our theoretical resistance is not a good approximation of the real resistance, due to the limitations of Michell's wave resistance at low Froude numbers~\cite{Gotman}).   
\begin{center}
\begin{figure}[!ht]
\includegraphics[height=5.5cm]{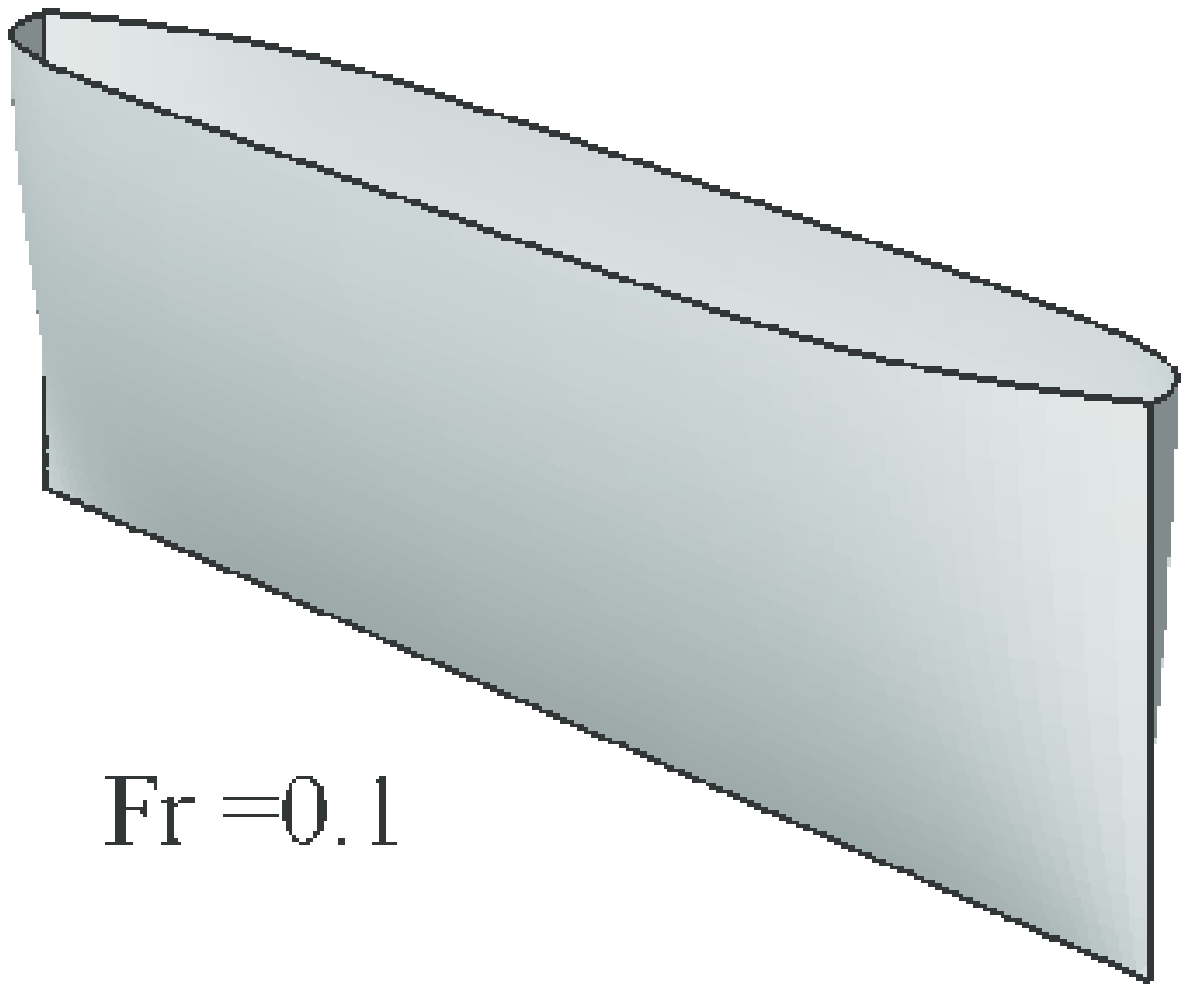}
\includegraphics[height=5.5cm]{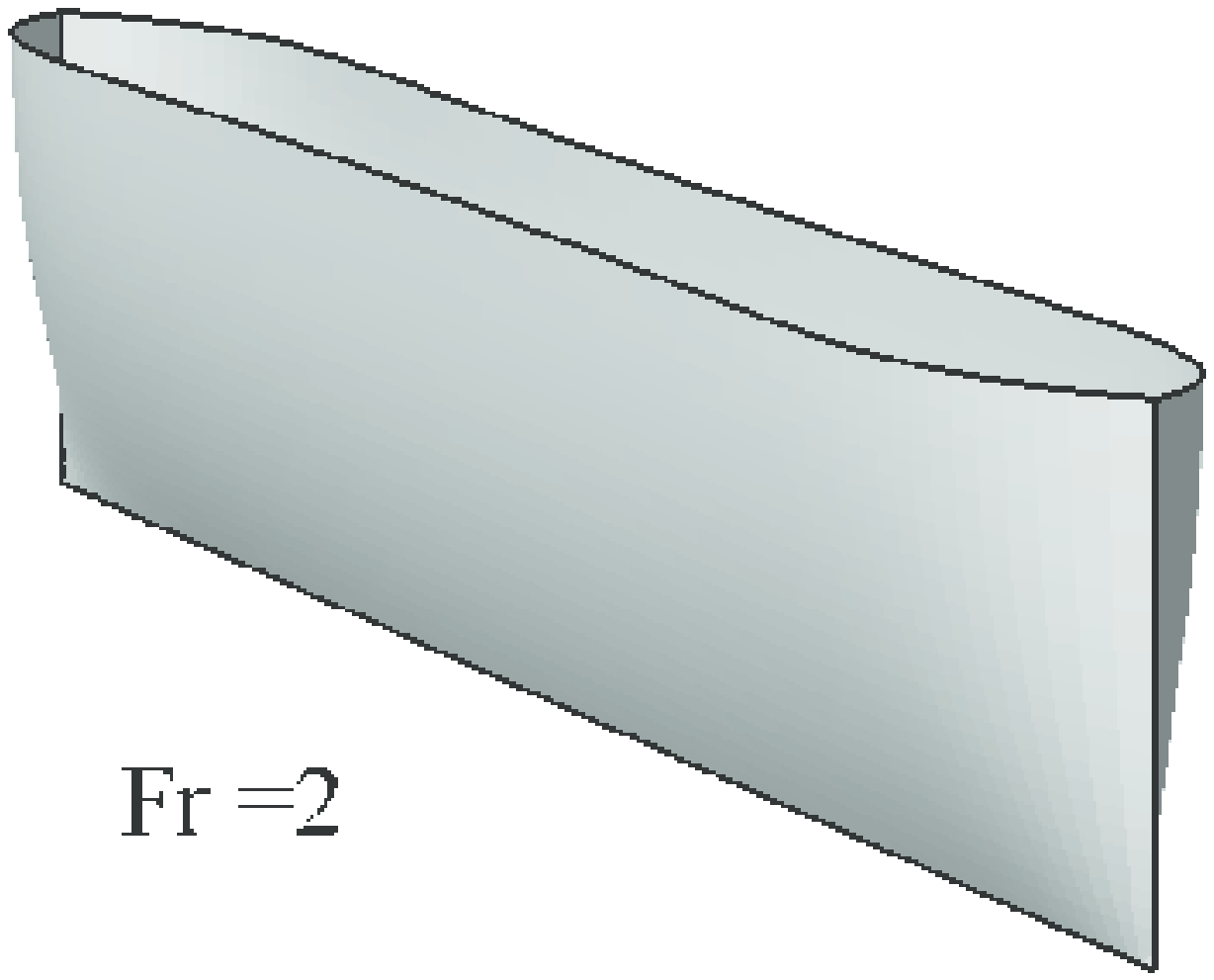} \\
\includegraphics[height=4.5cm]{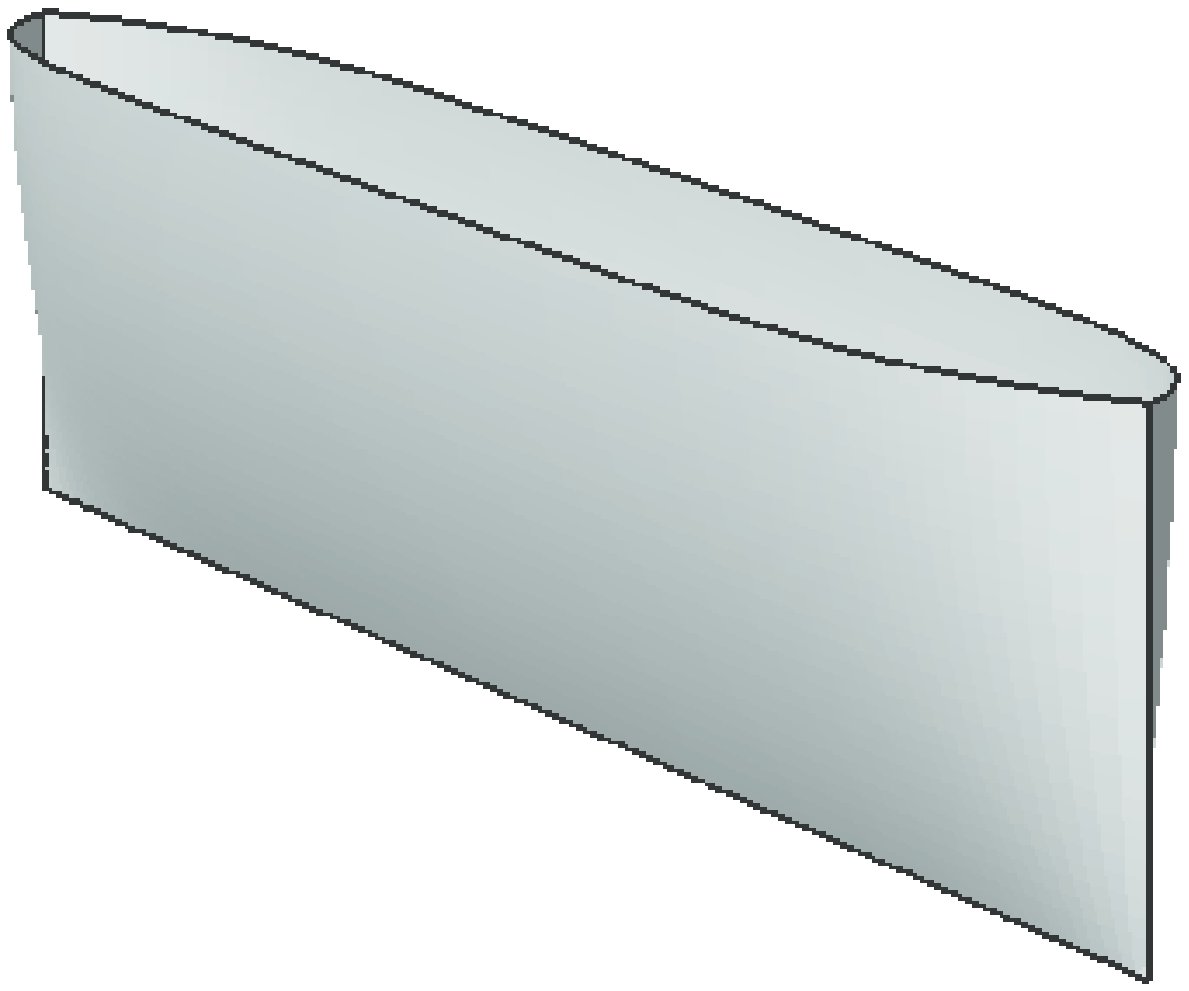} 
\caption{\label{res-3} Top row: ship hull optimization for high and low Froude numbers. Bottom figure: ship hull optimization without wave resistance (optimization of the viscous drag).}
\end{figure}
\end{center}
 
In the intermediate regimes (here $Fr \in [0.1, 1]$) in which the wave resistance is non-negligible, we observe various hull shapes depending on the length Froude number (see Figure~\ref{res-4}). Here, for $Fr$ close to $0.6$ we observe that the optimal hull features a bulbous bow, very similar to the ones that are usually designed for large sea ships~\cite{Inui}. For $Fr\in [1,2]$ the optimized hull varies continuously from a form presenting a small bulbous bow to a shape where the wave resistance is negligible.  
\begin{center}
\begin{figure}[!ht]
\includegraphics[height=5.3cm]{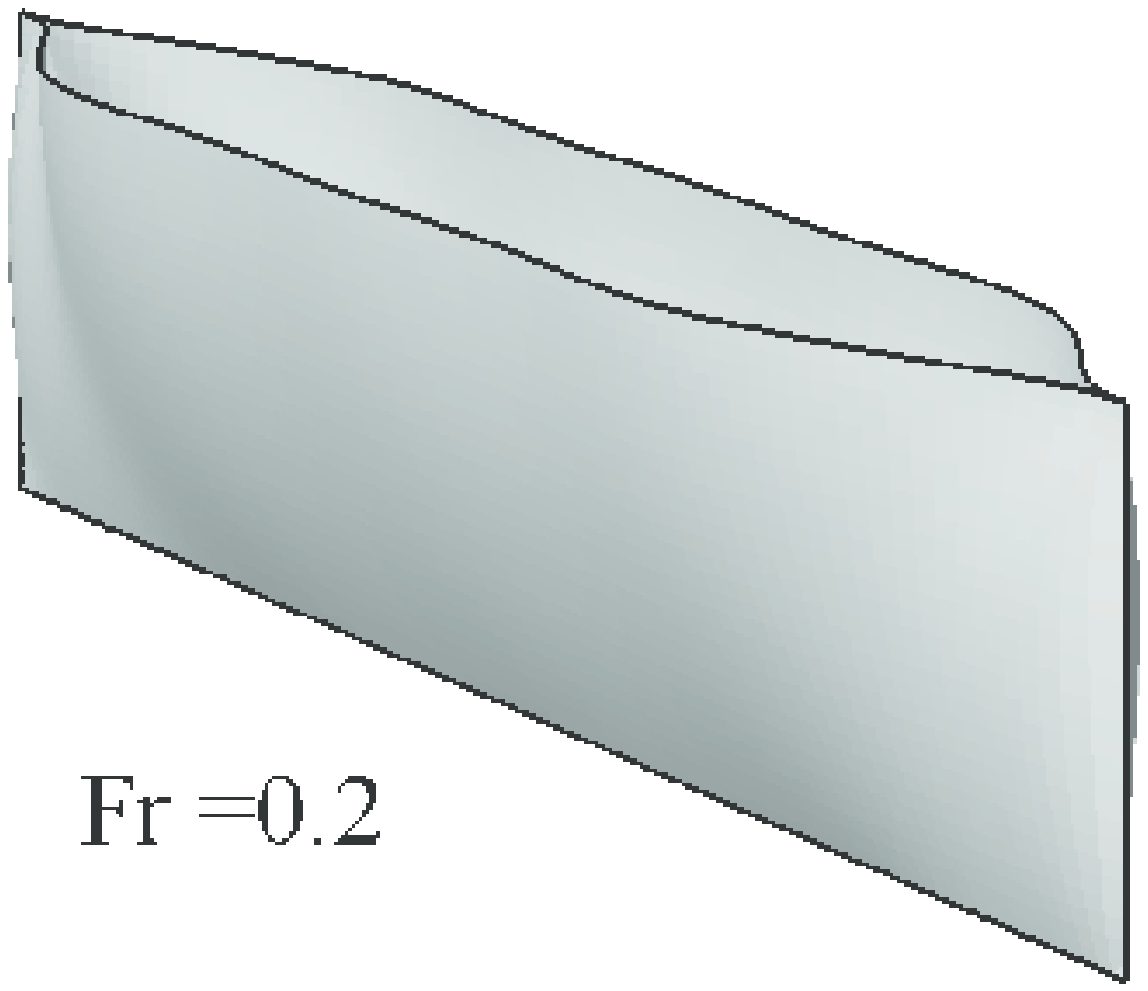}
\includegraphics[height=5.3cm]{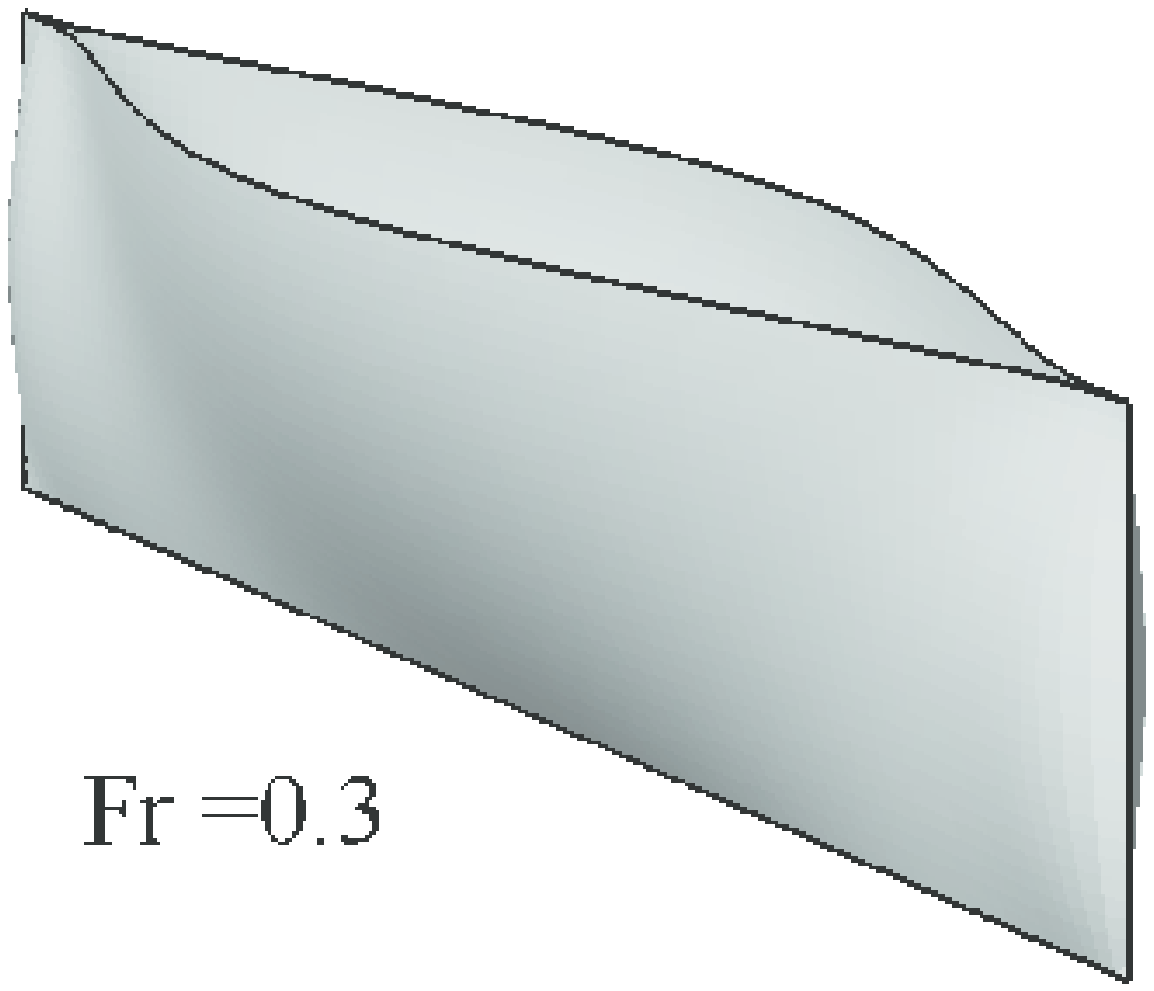}
\includegraphics[height=5.3cm]{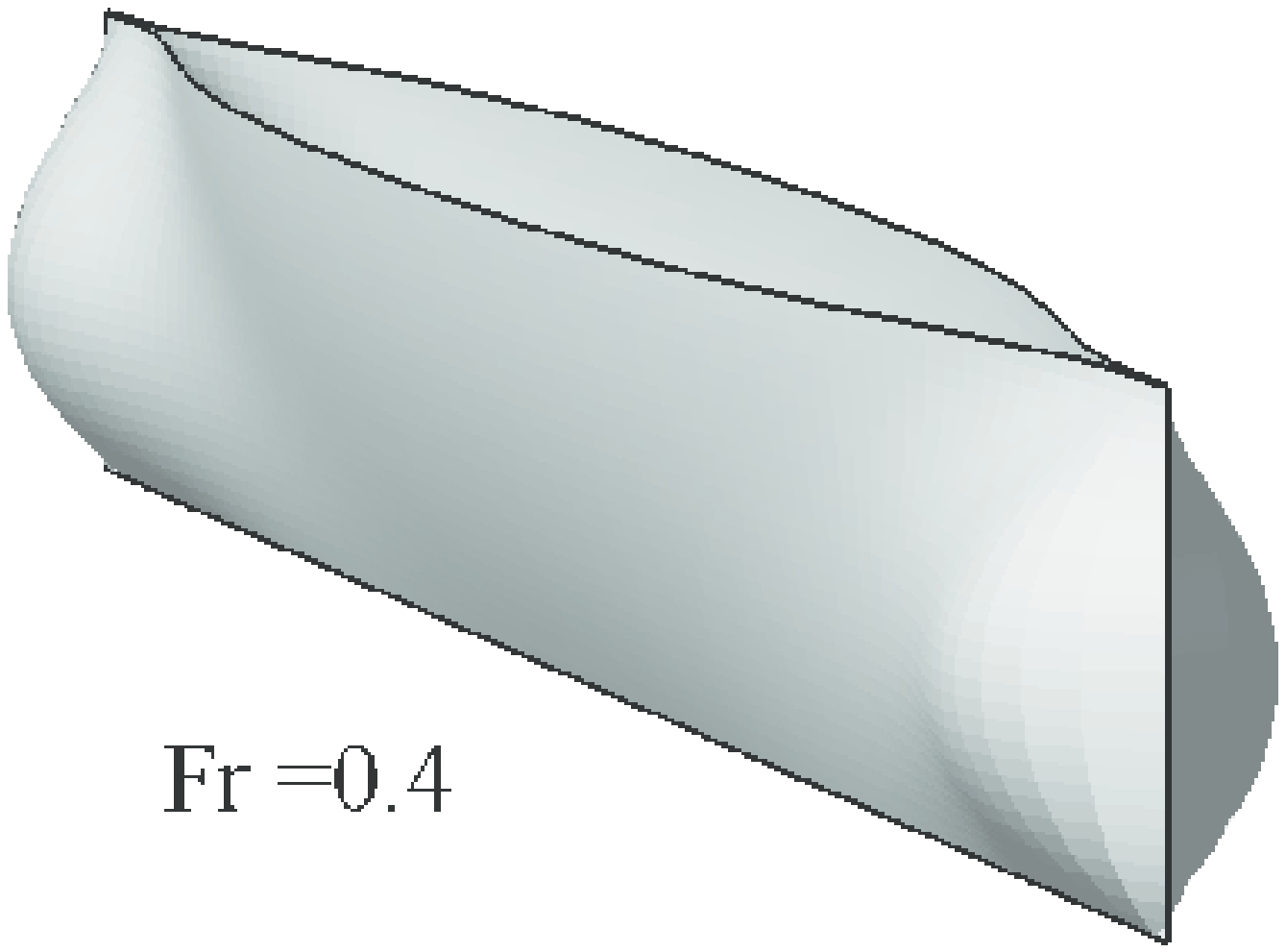}
\includegraphics[height=5.3cm]{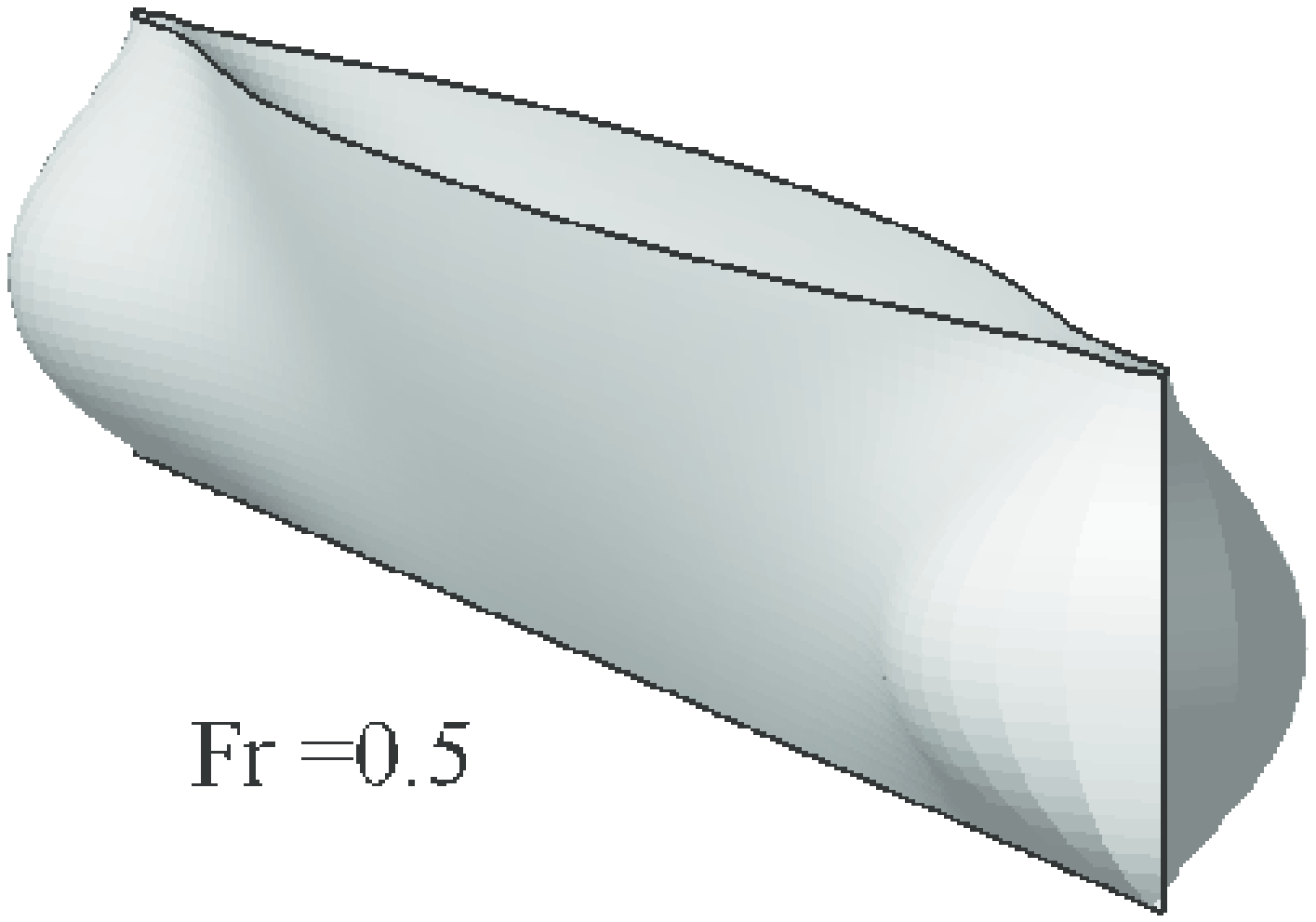}
\includegraphics[height=5.3cm]{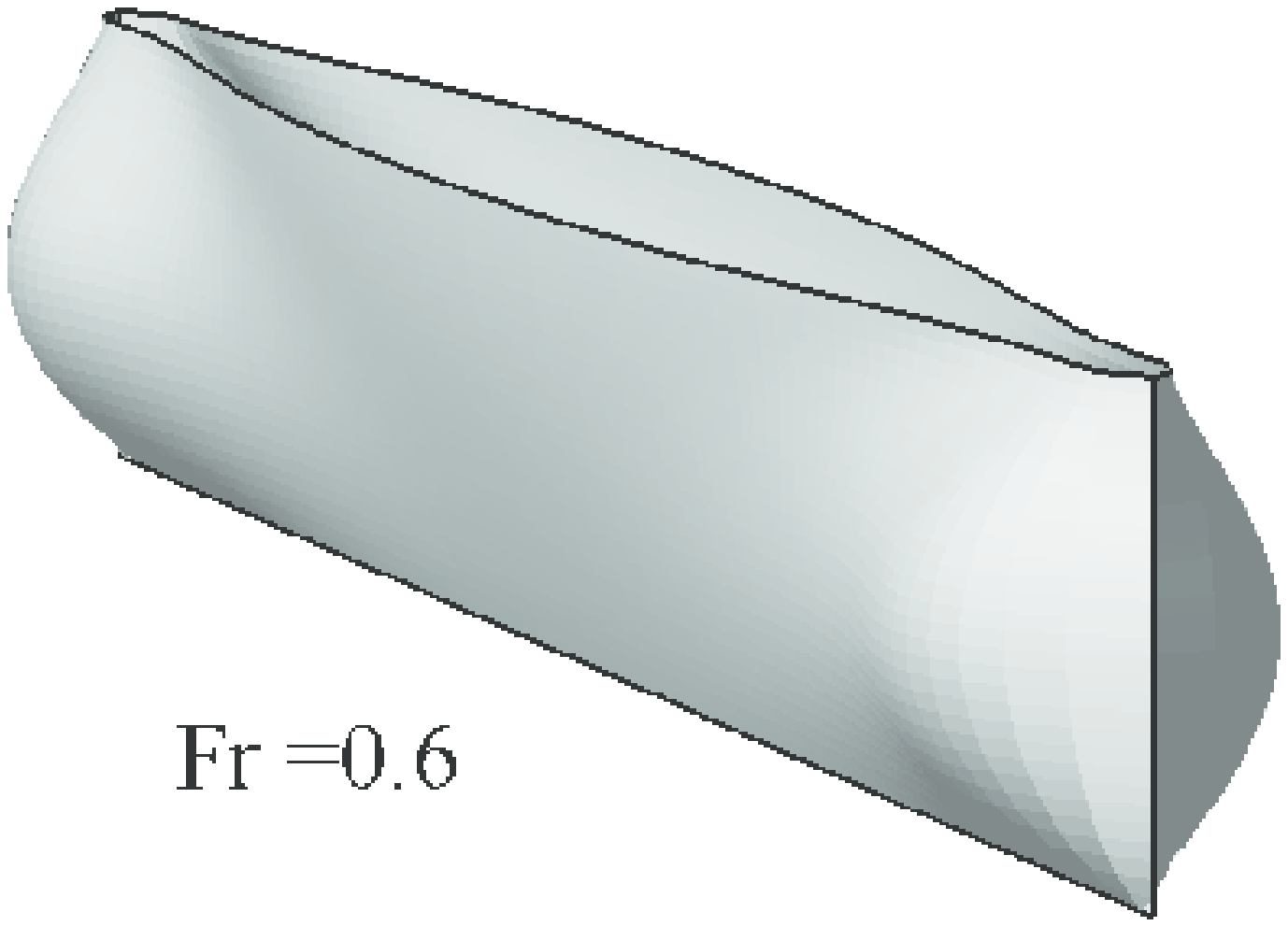}
\includegraphics[height=5.3cm]{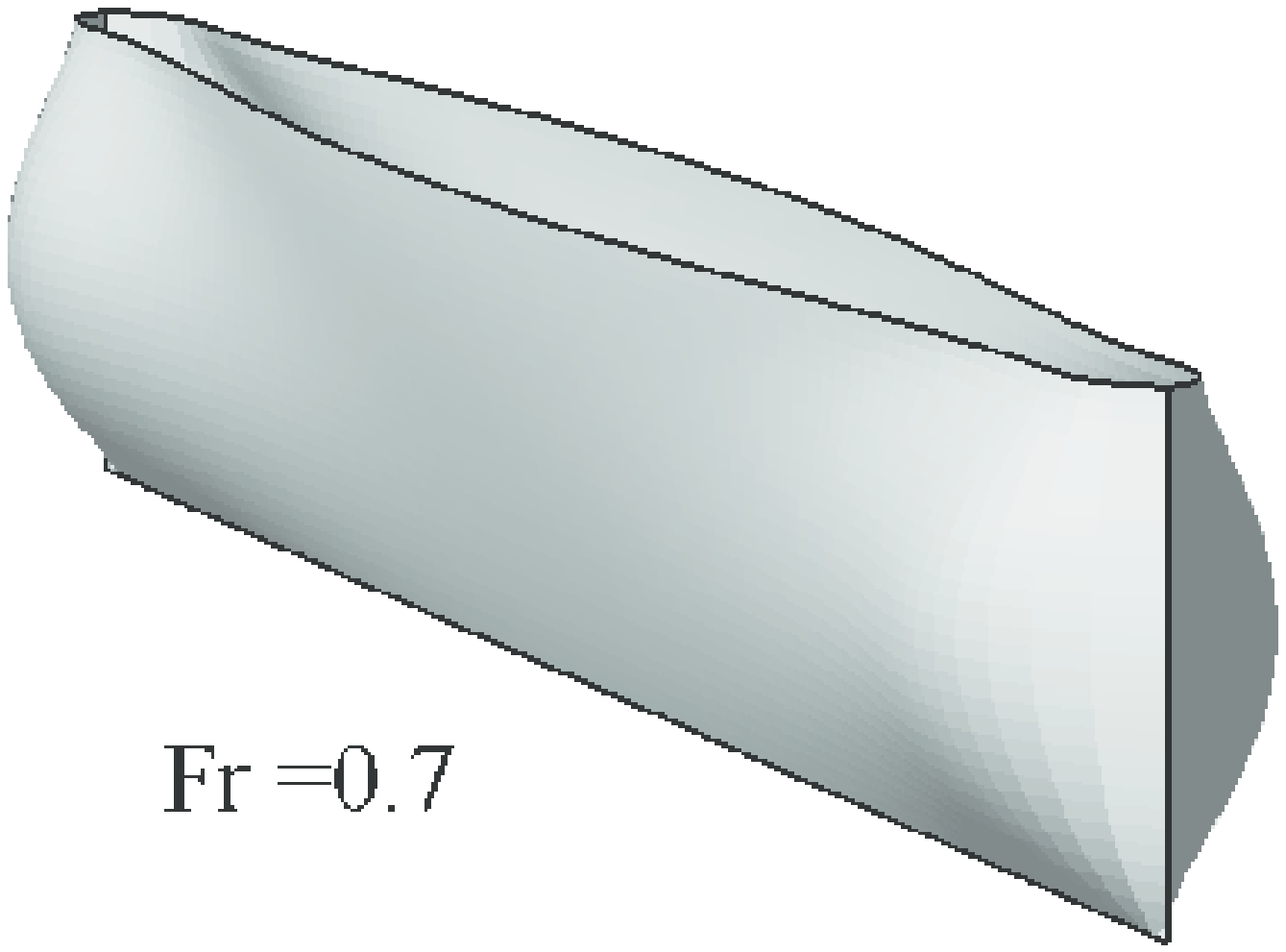}
\caption{\label{res-4}Ship hull optimization for moderate Froude numbers.}
\end{figure}
\end{center}
Note that this bulbous bow appears for Froude numbers values that usually produce the largest wave resistance for a standard hull such as the Wigley hull (see Figure~\ref{res-5}, plain line). In Figures~\ref{res-5a}-\ref{res-5},   we observe that the optimized hull for a given velocity is not optimal for every velocities. A Wigley hull can be a better solution for some values of $Fr$. For the comparison,  we have used here a Wigley hull with a parabolic cross section, i.e. 
$$f(x,z)=\frac{B}{2}(1-\frac{4x^2}{L^2})(1-\frac{z^2}{T^2}),$$
where $B$ is such that
$$V=\int_{-L/2}^{L/2}\int_0^Tf(x,z)dxdz=\frac{2}{9}BLT=0.03\mathrm{m}^3.$$ 
\begin{center}
\begin{figure}[!ht]
\includegraphics[width=10cm]{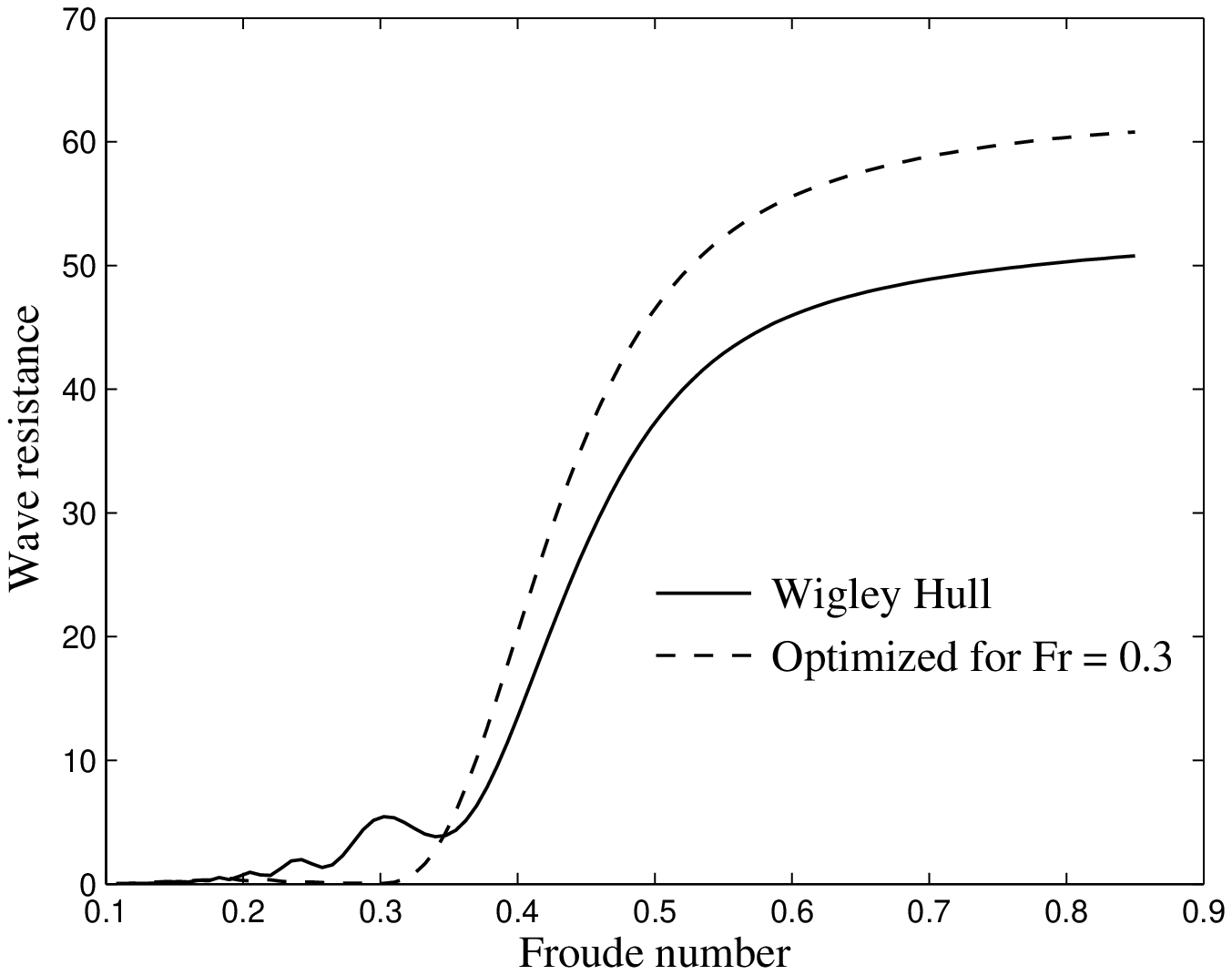}
\caption{\label{res-5a}Comparison with a Wigley hull}
\end{figure}
\end{center}
\begin{center}
\begin{figure}[!ht]
\includegraphics[width=10cm]{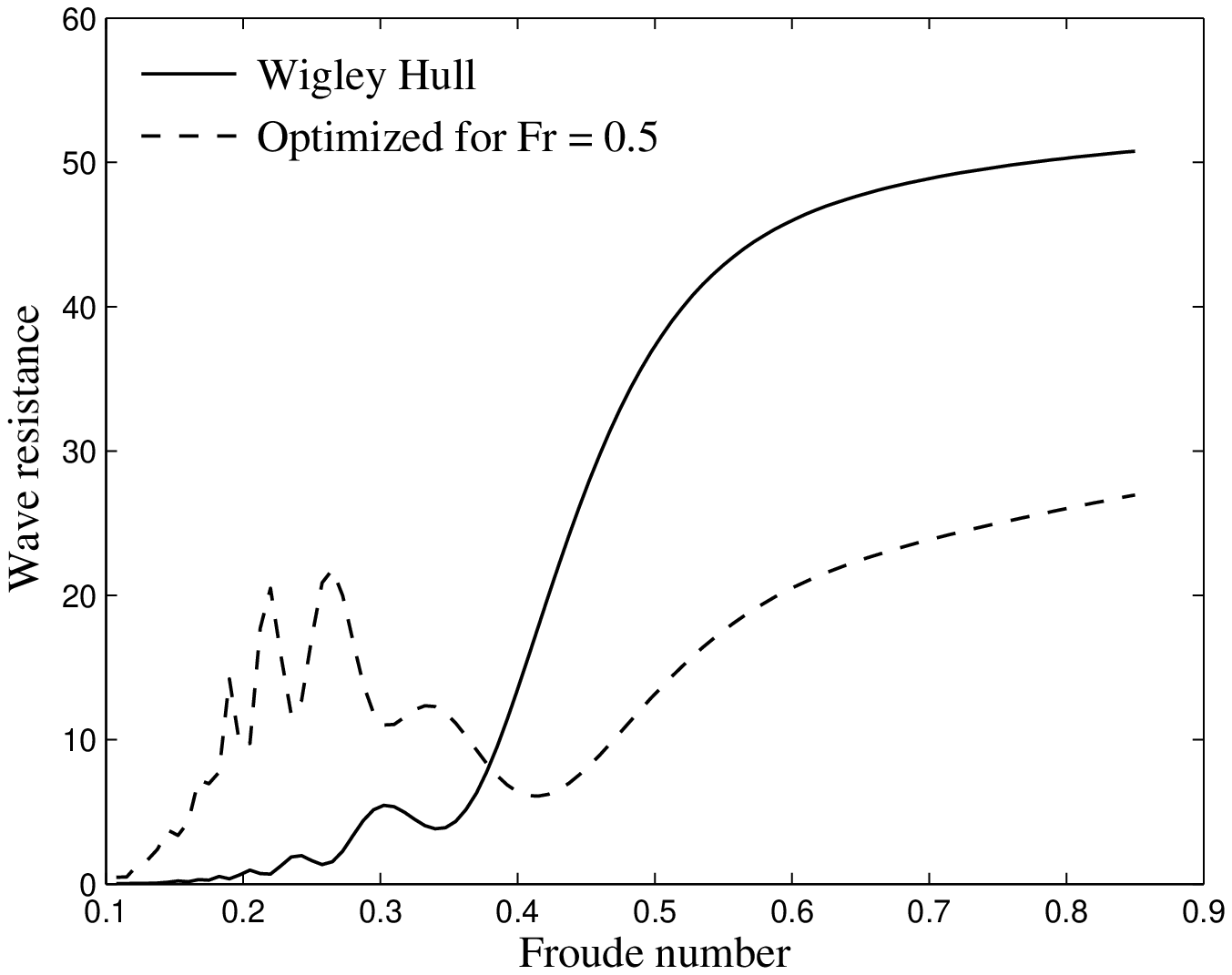}
\caption{\label{res-5}Comparison with a Wigley hull}
\end{figure}
\end{center}

\section{Conclusion and perspectives}
In this paper we presented both a theoretical and numerical framework for the optimization of ship hull in the case of unrestricted water, in which the Mitchell's integral is valid for the prediction of the wave resistance. We have shown the well-posedness of the problem when adding a regularising term that can be interpreted physically as a model of viscous resistance. Some numerical calculations have shown some features predicted in the theoretical work such as the most-likely ill-posedness of the optimization problem when considering only the wave resistance as our objective function and the fact that one could reduce the number of degrees of freedom in our problem by working on the (smaller) space of hulls that produce a non-zero wave resistance (although an expression of a basis of this space seems a non-trivial). Further numerical calculations have shown some common features of ship design such a the use of a bulbous bow to reduce the wave resistance. \\

\section*{Acknowledgements}
The authors are thankful to the ``Action Concert\'ee Incitative: R\'esistance de vagues (2013-2014) of the University of Poitiers and to the ``Mission Interdisciplinaire  of the CNRS (2013)'' for financial support. The authors also acknowledge the group ``Phydromat'' for stimulating discussions.

{}

\end{document}